\documentclass{amsart}
\usepackage{amsmath,amscd,amssymb}
\usepackage[english]{babel}
\usepackage[applemac]{inputenc}

\newtheorem{sz}{Satz}[section]
\newtheorem{Th}[sz]{Theorem}
\newtheorem{Prop}[sz]{Proposition}
\newtheorem{Rem}[sz]{Remark}
\newtheorem{Cor}[sz]{Corollary}
\newtheorem{Def}[sz]{Definition}
\newtheorem{Lem}[sz]{Lemma}
\newtheorem{Ex}[sz]{Exemple}
\def\cal{\mathcal}
\def\bb{\mathbb}      
\def\a{\alpha }
\def\b{\beta }

\def\f{\varphi }

\def\G{\Gamma }

\def\part{\partial }

\def\s{\sigma }

\def\dim{{\rm dim\; }}

\def\id{\mathrm{id}}
\def\subsetint{{\ {\subset}\hskip -2.2mm{\raisebox{.25ex}
{$\scriptscriptstyle\subset$}}\ }}
\def\resto#1#2{{
#1\hskip 0.4ex\vline_{\hskip 0.2ex\raisebox{-0,2ex}
{{${\scriptstyle #2}$}}}}}

\def\textmap#1{\mathop{\vbox{\ialign{
                                 ##\crcr
     ${\scriptstyle\hfil\;\;#1\;\;\hfil}$\crcr
     \noalign{\kern 1pt\nointerlineskip}
     \rightarrowfill\crcr}}\;}}
\def\qmod#1#2{{\hbox{}^{\displaystyle{#1}}}\!\big/\!\hbox{}_{
\displaystyle{#2}}}

\begin{document}
\def\C{\mathbb C}
\def\Q{\mathbb Q}
\def\R{\mathbb R}
\def\Z{\mathbb Z}
\def\vol{\mathrm{vol}}
\def\Hom{\mathrm{Hom}}
\def\Pic{\mathrm{Pic}}
\title {Infinite bubbling in non-K\"Ahlerian geometry}
\author{Georges Dloussky}
\address{Georges Dloussky, LATP, CMI, Universit\'e de Provence, 39 rue F. Joliot-Curie, 13453 Marseille Cedex 13, France}
\email{dloussky@cmi.univ-mrs.fr}
\author{Andrei Teleman}
\address{Andrei Teleman, LATP, CMI, Universit\'e de Provence, 39 rue F. Joliot-Curie, 13453 Marseille Cedex 13, France}
\email{teleman@cmi.univ-mrs.fr}
\begin{abstract}

In a holomorphic family $(X_b)_{b\in B}$ of non-Kählerian  compact manifolds, the holomorphic curves representing a fixed 2-homology class do not form a proper family in general. The deep source of this fundamental difficulty in non-Kähler geometry is the {\it explosion of the area} phenomenon: the area of a curve $C_b\subset X_b$ in a fixed 2-homology class   can diverge as $b\to b_0$. This phenomenon occurs frequently in the deformation theory of class VII surfaces. For instance it is well known that any minimal GSS surface $X_0$ is  a degeneration  of a 1-parameter family of simply blown up primary Hopf surfaces $(X_z)_{z\in D\setminus\{0\}}$, so one obtains non-proper families of exceptional divisors $E_z\subset X_z$ whose area diverge as $z\to 0$. Our main goal is to study in detail this non-properness phenomenon in the case of class VII surfaces. We will prove that, under certain technical assumptions, a lift $\widetilde E_z$ of $E_z$ in the universal cover $\widetilde X_z$ does converge to an effective divisor $\widetilde E_0$ in $\widetilde X_0$, but this limit divisor is not compact. We prove that this limit divisor is always bounded  towards the pseudo-convex end of $\widetilde X_0$ and that, when $X_0$ is a minimal surface with global spherical shell, it is given by an infinite series of {\it compact} rational curves, whose coefficients can be computed explicitly. This phenomenon -- degeneration of a family of compact curves to  an infinite union of compact curves -- should be called {\it infinite bubbling}.

We believe that such a decomposition result holds for any family of class VII surfaces whose generic fiber is a blown up primary Hopf surface. This statement would have important consequences for the classification of class VII surfaces.

\end{abstract}

\thanks{The authors wish to thank the unnamed referee for the careful reading of the paper and for his useful suggestions.}

\maketitle

MSC:  32Q65, 32Q57, 32Q55

\tableofcontents

\section{Introduction: What happens with the exceptional curves in a deformation of class VII surfaces?}

\subsection{Properness properties in Kählerian and symplectic geometry}

We will use the following notation: for a complex manifold $X$ and a class $c\in H_{2d}(X,\Z)$ we denote by ${\cal B}^c(X)$  the Barlet space \cite{Ba} of $d$-cycles in $X$ representing the class $c$.

Let  $p:{\cal X}\to B$  be a family of compact $n$-dimensional complex manifolds parameterized by a  complex manifold $B$.  In other words, $p$ is a proper holomorphic submersion with connected fibers. We denote by $\underline{H}_m$,  $\underline{H}^m$ the locally constant sheaves (coefficient systems) on $B$ which are associated with the presheaves $B\supset U\mapsto H_m(p^{-1}(U),\Z)$, respectively $B\supset U\mapsto H^m(p^{-1}(U),\Z)$. Note that the sheaf $\underline{H}^m$ coincides with the $m$-th direct image $R^mp_*(\underline{\Z})$ of the constant sheaf $\underline{\Z}$ on ${\cal X}$. 
We fix   $e=(e_b)_{b\in B}\in H^0(B,\underline{H}_{2d})$, and we denote by ${\cal B}^e_B({\cal X})$ the relative Barlet cycle space associated with $e$. This space is a closed subspace of the cycle space ${\cal B}({\cal X})$ and, as a set, it coincides with the union $\cup_{b\in B} {\cal B}^{e_b}(X_b)$. The obvious map 
$${\cal B}^e(p):{\cal B}^e_B({\cal X})\to B$$
 is holomorphic and its (reduced) fibers are precisely the spaces  ${\cal B}^{e_b}(X_b)$.   A fundamental consequence of Bishop compactness theorem states that in the Kählerian framework the projection ${\cal B}^e(p):{\cal B}^e_B({\cal X})\to B$ is also proper.  More precisely:
\begin{Prop}
Let  $H$ be a Hermitian metric on ${\cal X}$ such that the restrictions to the fibers  of the corresponding Kähler form $\Omega_H\in A^{1,1}({\cal X},\R)$ are all closed\footnote{We can always construct such  a Hermitian metric in a neighborhood of a fixed fiber $X_{b_0}$ if this  fiber is Kählerian. It suffices to consider a smooth family of Kähler metrics $h_b$ on $X_b$ (using the openness property of the Kähler condition) for $b$ close to $b_0$, and to add the pull-back of a Hermitian metric on the base.}. 
Then the map ${\cal B}^e(p):{\cal B}^e_B({\cal X})\to B$ is proper.

\end{Prop}
\begin{proof}
The total space of the locally constant sheaf $\underline{H}^{2d}\otimes\R=R^{2d} p_*(\underline{\R})$ is a real vector bundle on $B$, and the assignment   $b\mapsto [\resto{\Omega_H^d}{X_b}]$ defines a smooth section $\chi$ of this bundle.   For a cycle $C\in B^{e_b} (X_b)$ one has
$$\vol_H(C)=\int_C \Omega_H^d= \langle \chi(b),e_b\rangle\ .
$$
Therefore the volume map $\vol_H:{\cal B}^e_B({\cal X})\to\R_{\geq 0}$ can be written as $ {\cal B}^e(p)^*\langle  \chi,e\rangle$, so it is bounded on ${\cal B}^e(p)^{-1}(K)$ for every  compact set $K\subset D$.  It suffices to apply the relative-compactness criterion given by Theorem 1 of \cite{Ba}.
\end{proof}
This properness result can be intuitively interpreted as follows: 
\\ \\
{\it The cycles in a fixed homology class of a 
compact complex manifold $X_b$  cannot disappear as  $X_b$ tends to a limit compact Kähler  manifold $X_{b_0}$.}
\\ \\
The properness property has a particular signification
 in the case $n=2$, $d=1$.
  Indeed, in this case  the Barlet
 space ${\cal B}^e(X_b)$ is just a moduli space of effective divisors representing the class $e$, and this moduli space can be identified with the moduli space of Seiberg-Witten monopoles associated with the $e$-twisted canonical $\mathrm{Spin}^c$-structure of $X_b$ and a suitable perturbation of the monopole equation (see \cite{W}, \cite{OTe}, \cite{Te4}). Therefore, in this case, the properness property above is a consequence of the general properness property  -- with respect to the space of parameters -- for moduli spaces of Seiberg-Witten monopoles, which is  crucial  for proving that  Seiberg-Witten invariants are well-defined. \\

A similar  properness result --  based on a version of Gromov compactness theorem -- 
 holds in symplectic geometry:  let ${\cal X}\to B$ be a proper smooth map with connected fibers in the differentiable category, suppose that $\Omega$ is a smooth 2-form on ${\cal X}$ whose restrictions to the fibers $\omega_b:=\resto{\Omega}{X_b}$ are all  symplectic forms (in particular closed), and consider a smooth family $(J_b)_{b\in B}$ of almost complex structures on the fibers, such that $J_b$ is compatible with $\omega_b$ for any $b\in B$, and a section   $e=(e_b)_{b\in B}\in H^0(B,\underline{H}_2)$. Let $(b_n)_{n\in\mathbb{N}^*}$ be a sequence in 
$B\setminus\{b_0\}$ converging to $b_0$ and consider, for every $n\in\mathbb N$, an almost holomorphic map 
$$\varphi_n:(\Sigma,j_{n})\to (X_{b_n},J_{b_n})$$
 representing   
 $e_{b_n}\in H_2 (X_{b_n},\Z)$, where $j_n$ is an almost
 holomorphic structure on a differentiable closed surface $\Sigma$. Since 
the map $b\mapsto[\omega_b]\in H^2(X_b,\R)$ defines a smooth section in the vector bundle $\underline{H}^2\otimes\R$, we get a bound of the area sequence 
$$\bigg(\int_{\Sigma} \varphi_n^*(\omega_{b_n})\bigg)_{n\in\mathbb N}=
\big(\langle [\omega_{b_n}], e_{b_n}\rangle\big)_{n\in\mathbb N}\ .$$
   Therefore 
 (using the terminology and the main result of \cite{Pa}) there exists a subsequence $(b_{n_k})$ of $(b_n)_{n\in\mathbb{N}}$ such that that the sequence $(\varphi_{n_k})_{k\in\mathbb N}$ converges  in the ${\cal C}^m$-topology  to a  {\it  cusp pseudo-holomorphic curve}
$$\varphi:\cup _l \Sigma_l\to X_{b_0}\ .
$$
 
A similar convergence result  can be proved for sequences 
of cusp curves. In other words, in 
symplectic geometry -- as in Kählerian
 geometry -- pseudo-holomorphic curves representing a fixed homology class  can degenerate, but {\it they cannot disappear in a limit process}. An important particular case is the one when $j_n=j$ is independent of $n$. In this case it is known that the domain $\cup _l \Sigma_l$ of the limit cusp curve  is obtained from $\Sigma$ by adding a union of {\it bubble trees} of 2-spheres. In particular, when $(\Sigma, j)=\bb P^1$, all the irreducible components of the limit cusp curve will be 2-spheres.

These properness results and their natural generalizations for perturbed pseudo-holomorphic curves, play a crucial role in proving that (different versions of)  Gromov-Witten invariants are well-defined. 
\\

For deformations of compact non-Kählerian manifolds the analogue properness results do not hold.  {\it The properness property fails even in the cases when all cycle spaces ${\cal B}^e(X_b)$ are compact.}   This is one of the main difficulties in non-Kählerian geometry, which has dramatic consequences, for instance: in non-Kählerian geometry one cannot use enumerative methods of curves to construct deformation invariants, and this {\it even if one  considers only moduli spaces of curves which are finite for all manifolds in the considered deformation class!}    In non-Kählerian geometry  holomorphic curves  representing a given homology  class {\it can} disappear in a limit process, because the area of the curves in a given homology class is not a priori bounded, so it can explode to infinity when one approaches a  fixed fiber $X_{b_0}$. Probably the most convincing examples which illustrates this difficulty (and its dramatic consequences) appear in the deformation theory of class VII surfaces.

\subsection{Deformations of class VII surfaces. What happens with the exceptional curves?}

\label{DefClassVII}

Let $X$ be a  class VII surface with $b_2:=b_2(X)>0$
 admitting a global spherical shell (a GSS), see section \ref{GSSsurfaces}
 for details. We will call such a surface a GSS surface.
For such a surface  one has $\pi_1(X,x_0)\simeq\Z$, and the universal cover $\widetilde X$ has two ends, a pseudo-convex end and a pseudo-concave end. Taking a non-separating  strictly pseudo-convex embedded 3-sphere $\Sigma\subset X$, consider a lift $A$ of $X\setminus\Sigma$ in $\widetilde X$ and note that $\bar A$ is a compact  surface bounded by two lifts  of $\Sigma$ to $\widetilde X$; the pseudo-convex end of $\widetilde X$ corresponds to the pseudo-convex boundary component of $A$ \footnote{Note that the terminology used here for the two ends of $\widetilde X$ is different from the one used in \cite{D1}, but agrees with the terminology used in the theory of complete, non-compact complex manifolds. }. 

For such a surfaces  one also has  $H_2(X,\Z)\simeq H^2(X,\Z)\simeq\Z^{b_2}$, and the intersection form on this group is standard, i.e. there exists a basis $(e_0,\dots,e_{b_2-1})$ in $H_2(X,\Z)$ such that $e_i\cdot e_j=-\delta_{ij}$. Changing signs  if necessary one can assume that 
$$-c_1(X)=c_1({\cal K}_X)=\sum_{i=0}^{b_2-1} e_i\ ,$$
and  the basis obtained  in this way is unique up to order. 

By results of Kato and Dloussky it is well-known that any minimal GSS surface   $X_0$    is deformable in blown up primary Hopf surfaces. More precisely, there exists a holomorphic family ${\cal X}\to D$ parametrized by the unit disk $D\subset\C$  having $X_0$ as central fiber such that $X_z\simeq  [\hat H_z]_{P(z)}$ is biholomorphic to  a primary Hopf surface $H(z)$ blown up in a finite set of simple points $P(z)=\{p_0(z),\dots,p_{b_2-1}(z)\}$. Therefore, for any $z\ne 0$, the fiber $X_z$ has $b_2$ exceptional curves $E_{0,z},\dots,E_{b_2-1,z}$ representing the classes $e_i$. On the other hand the central fiber $X_0$ is minimal  hence, by a result of Nakamura (see \cite{N2}, Lemma 1.1.3] ), it does not admit any  effective divisor at all representing  a class $e_i$. This shows that the rational curve $E_{i,z}$ has no limit at all as $z\to 0$, even if one  admits cusp curves (and bubbling trees) in the limit process. 

Note that, for $z\ne 0$, the exceptional curve $E_{i,z}$ is a regular point in the moduli space of divisors representing the class $e_i$ on $X_z$, so if one tries to define a Gromov-Witten type "invariant" of $X_z$ by counting the rational curves in this  homology class, the result will be 1. But the same "invariant" will be 0 for $X_0$. 

Similar phenomena appear for deformations of minimal GSS surfaces. For instance, there exists a family ${\cal X}\to D$ of minimal GSS surfaces with $b_2=1$ whose generic fiber $X_z$ ($z\ne 0$) has only one irreducible curve $C_z$, which  is   homologically trivial,  whereas $X_0$ has only one irreducible curve $D$, which represents the class $-e_0$. In this family all curves are singular rational curves with a simple node.  In both cases one can study explicitly the evolution of the  area of the curves  $E_{i,z}$, $C_z$ and see that their area  tends to infinity as $z\to 0$. 

 The conclusion is clear: one cannot hope to define deformation invariants for class VII surfaces by counting holomorphic curves in a given homology class. This is one of the major difficulties in the classification of class VII surfaces, because the structure  of curves on these surfaces is related to the fundamental classification problem.  Indeed, by the main result of \cite{DOT} one knows that the global spherical conjecture  (which, if true, would complete the classification of class VII surfaces) reduces to the conjecture:
\vspace{2mm}\\
{\bf C:} {\it Any minimal class VII surface $X$ with $b_2(X)>0$ has $b_2(X)$ rational curves.
}
\vspace{2mm}\\
Therefore, proving existence of curves is a fundamental 
problem in the theory of class VII surfaces.  Unfortunately
 the examples above show that one cannot hope to prove such existence 
results using invariants of Gromov-Witten type. Note however that, taking into account
 this difficulty (the lack of Gromov-Witten type invariants), the
 conjecture above 
becomes quite intriguing. Indeed, 
this conjecture (which is true for all GSS surfaces) implies that the total number of rational curves {\it is} an invariant for  class VII surfaces, although this invariant cannot be obtained  as the sum of Gromov-Witten type invariants defined for individual homology classes (because, as we have just seen, the number of rational curves in an individual class is {\it not} an invariant). In other words, although the homology classes represented by rational curves change in a holomorphic family, the total number of rational curves remains always constant  in deformations of known class VII surfaces. In particular, when (in a limit process) a homology class loses a rational curve, there should be always another  class which gets one. This  ``global compensation phenomenon"  between different moduli spaces of rational curves is not understood yet, so conjecture ${\bf C}$ is not known yet even for surfaces which are degenerations of the known GSS surfaces.
\vspace{2mm}

Although the cause of the ``curve vanishing " (non-properness) phenomenon is clear -- explosion of the area in the limit process -- one can still wonder when contemplating the first family described above: \\
\\
\centerline {\it ``What happens with the exceptional curves  $E_{i,z}$ when $z\to 0$?"}
\\

 The authors have  many times been  asked this question when they gave talks on class VII surfaces in the past. The easy answer we have always given 
\\
\\
\centerline {\it ``They just disappear, because their area explodes",} 
\\
\\although perfectly correct,  has  never seemed to fully satisfy the audience.   Therefore we decided to investigate in detail this phenomenon trying to understand what happens geometrically with the exceptional curve $E_{i,z}$ as it gets bigger and bigger. 

We came to the idea of an "infinite bubbling  tree" empirically,  when the second author noticed that recent results  of the  first author \cite{D4} have an intriguing consequence: Although for a minimal GSS surface   $X_0$ the homology class  $e_i$  is not represented by an effective divisor (as we saw above), any   of its lifts $\widetilde e_i$ in the universal cover 
 $\widetilde X_0$ is represented -- in the homology with closed supports -- by an infinite series of compact curves.  Moreover, in every case   we obtained a well  defined series representing the given lift, and we saw that always this series escapes to  infinity towards the pseudo-concave end of $\widetilde X_0$. Therefore we came to idea to prove that (although $E_{i,z}$ accumulates to the whole $X_0$ as $z\to 0$ by the Remmert-Stein theorem) a fixed lift $\widetilde E_{i,z}$ of $E_{i,z}$ in the universal covering $\widetilde X_z$ of $X_z$ behaves in a controllable way as $z\to 0$, namely  that
\begin{Th} \label{firstTh}
The family of rational curves $(\widetilde E_{i,z})_{z\in D\setminus\{0\}}$  fits in a flat family $\widetilde{\cal E}_i$ of effective divisors  over the whole base $D$. In particular  this family has a limit as $z\to 0$, which is given by an infinite series of compact curves which escapes to  infinity towards the pseudo-concave end of $\widetilde X_0$. 
\end{Th}

The main purpose of this article is to prove this result.  More precisely we will prove that the limit given by the theorem is  precisely the series representing $\widetilde e_i$ we discovered empirically. 
 We called this phenomenon "infinite bubbling", and we believe that this happens in all deformations of class VII surfaces when area explosion occurs.  
In fact we will prove a more general result:
\begin{Th} \label{secondTh} Let ${\cal X}\to \Delta$ be any holomorphic family of GSS  surfaces  with $b_2>0$, parameterized by the standard ball $\Delta\subset\C^r$.  Then for any lift $\widetilde e_i$  of a homology class $e_i$ to the universal cover $\widetilde{\cal X}$, there exists   an effective divisor $\widetilde {\cal E}_i\subset \widetilde {\cal X}$ flat over $\Delta$ such that
\begin{enumerate} 
\item The fiber $\widetilde E_{i,z}\subset \widetilde X_z$ of $\widetilde {\cal E}_i$ over a point $z\in\Delta$ is a lift of the exceptional curve $E_i$ representing $e_i$, for every $z\in\Delta$ for which $X_z$ contains such an exceptional curve.
\item If  $X_z$ does not admit any exceptional effective divisor in the class $e_i$ then the corresponding fiber  $\widetilde E_{i,z}$ of $\widetilde {\cal E}_i$ is a series of compact curves which escapes to infinity towards the pseudo-concave end of $\widetilde X_z$. This divisor represents the image of the class $\widetilde e_i$ in  the Borel-Moore  homology group $H_2^{\rm BM}(X,\Z)$. 
\end{enumerate}
\end{Th}

Finally, we can generalize 
this result for deformations of an {\it unknown} minimal class VII surface which can be deformed in GSS surfaces:

\begin{Th} \label{thirdTh} Let ${\cal X}\to \Delta$ be a holomorphic family such that  $X_z$ is a GSS surface with $b_2>0$ for any $z\in\Delta\setminus A$, where $A\subset\Delta$ is an analytic subset of codimension $\geq 2$, and $X_a$ is minimal for every $a\in A$. 
Then for any lift $\widetilde e_i$  of a homology class $e_i$ to the universal cover $\widetilde{\cal X}$, there exists   an effective divisor $\widetilde {\cal E}_i\subset \widetilde {\cal X}$ flat over $\Delta$ such that
the fiber $\widetilde E_{i,z}\subset \widetilde X_z$ of $\widetilde {\cal E}_i$ over a point $z\in\Delta$ has the properties:
\begin{enumerate}
\item it represents the image of the class $\widetilde e_i$ in  the Borel-Moore  homology group $H_2^{\rm BM}(X,\Z)$ for every $z\in\Delta$, 
\item  is a lift of the exceptional curve $E_i$ representing $e_i$, for every $z\in\Delta$ for which $X_z$ contains such an exceptional curve.
\end{enumerate}
 \end{Th}
Unfortunately we did not succeed to prove that the second statement 
of Theorem \ref{secondTh} also holds in this case, so we cannot prove that, for $a\in A$, the limit divisor $\widetilde E_{i,a}$ is given by a series of {\it compact} curves. The difficulty is to rule out  the appearance of an open Riemann surface  in the limit. Such a result would be of crucial 
importance for the classification of class VII surfaces, because it will prove 
the main conjecture {\bf C} for any class VII surface which fits as the central fiber
 of a bidimensional family ${\cal X}\to\Delta\subset\C^2$ whose fibers 
$X_z$, $z\ne 0$ are GSS surfaces (see Theorem \ref{converse}).

\section{The non-minimality divisors}

A compact effective divisor $D$ on a (non-necessary compact) complex surface $X$ will be called {\it simply exceptional}  if there exists a {\it smooth} surface $Y$, a point $y\in Y$,  and a surjective holomorphic map $c:X\to Y$  such that $\resto{c}{X\setminus D}:  X\setminus D\to Y\setminus\{y\}$ is biholomorphic, and the fiber $c^{-1}(y)$  over $y$ is $D$ (regarded as a complex subspace of $X$). If this is the case, the pair $(X,D)$ can be obtained from $Y$ by iterated blowing up at $y$. More precisely  $(X,D)$ is the last term  $(X_n,D_n)$ in a sequence $(X_1,D_1),\dots,(X_n,D_n)$ ($n\geq 1$), where 
\begin{itemize}
\item $X_1$ is the blow up of $Y$ at $y$ and $D_1$ is the corresponding exceptional divisor,   
\item $X_k$ is obtained from $X_{k-1}$ by blowing up at a point $y_{k-1}\in D_{k-1}$ for $2\leq k\leq n$, and $D_k$ is the pre-image of $D_{k-1}$ in $X_k$.
\end{itemize}

  Therefore $D$ is the pre-image in $X_n$ of the exceptional curve of first kind $D_1\subset X_1$, and it is a tree of smooth rational curves containing at least one exceptional curve of the first kind. 
\begin{Rem} For  simply exceptional divisor $D\subset X$ one has 
$$(D,{\cal K}_X)=(D,{\cal N}_{D/X})=-1\ .$$
Using the duality isomorphism 
$PD:H_2(X,\Z)\simeq H^2_{\rm c}(X,\Z)$ we have $$D^2:=\langle PD([D]),[D]\rangle=-1\ .$$
\end{Rem}

Since the arithmetic genus $g_a(D)$ vanishes, one obtains  $h^0({\cal O}_D(D))-h^1({\cal O}_D(D))=0$ by Riemann-Roch theorem for embedded curves (\cite{BHPV}, p. 65). Therefore, since $\deg_D({\cal O}_D(D))=-1<0$, we get
$$h^0({\cal O}_D(D))=h^1({\cal O}_D(D))=0\ ,
$$  
so, denoting by $e\in H_2(X,\Z)$ the homology class represented by $D$, it follows that the Douady space ${\cal D}ou^e(X)$ of compact effective divisors representing this class  is 0-dimensional at $D$, and  $D$ is a smooth point in this space (see \cite{BF} p. 135).
\begin{Rem} \label{unicity-div} Let $X$ be a compact surface and $D$, $E$ simple exceptional divisors representing the same rational homology class $e\in H_2(X,\Q)$. Then $D=E$. 
\end{Rem} 
\begin{proof}  $(X,D)$ is the last term $(X_n,D_n)$ of a sequence of pairs $(X_1,D_1),\dots,(X_n,D_n)$ obtained from $(Y,y)$ as above. We prove the statement by induction on $n$. If $n=1$, $D$ is an exceptional curve of the first kind. Since $DE=-1$, it follows that $D$ is one of the irreducible components of $E$.   But the $\Q$-homology classes of the irreducible components of a simply exceptional divisor are linearly independent, so one can have $[D]_{\Q}=[E]_{\Q}$ if and only if $D=E$.

If $n>1$, let $D'$ be the exceptional curve of the first kind of the blowing up $X_n\to X_{n-1}$. One has obviously $D' D=0$, so $D' E=0$. Therefore either $D'$ is an irreducible component of $E$, or $D'\cap E=\emptyset$. In both cases the direct images $(p_n)_*(D)$, $(p_n)_*(E)$ via the contraction map $p_n:X_n\to X_{n-1}$ are simply exceptional divisor in the rational homology class $(p_n)_*(e)$, so they coincide by induction assumption. On the other hand one can write $D=p_n^*((p_n)_*(D))$, $E=p_n^*((p_n)_*(E))$, so $D=E$.
\end{proof}
The condition  "$X$ admits a  simply exceptional effective divisor $D$" is open with respect to deformations of $X$. More precisely: 
\begin{Prop}\label{open}
Let   $p:{\cal X}\to B$ be a holomorphic submersion with $\dim({\cal X})=\dim(B)+2$.  Let $D_0\subset  X_{b_0}$ be a simply exceptional divisor, and $e\in H_2({\cal X},\Z)$ its homology class in ${\cal X}$.   There exists a neighborhood ${\cal U}$ of $D_0$ in ${\cal D}ou^e({\cal X})$ such that 
\begin{enumerate}
\item \label{1} for every $D\in {\cal U}$, $p(D)$ is a singleton denoted $\{\mathfrak{p}(D)\}$ in $B$  and $D$ is a tree of rational curves in $X_{\mathfrak{p}(D)}$,  
\item \label{2}the induced map $\mathfrak{p}:{\cal U}\to B$ is a local biholomorphism at $D_0$,
\item $D$ is simply exceptional in $X_{\mathfrak{p}(D)}$ for every $D\in {\cal U}$
\end{enumerate}
\end{Prop}
\begin{proof} Let $V\subset B$ an open neighborhood of $b_0$ which is   the domain of a chart, and ${\cal V}\subset{\cal D}ou^e({\cal X}) $ the neighborhood of $D_0$ in ${\cal D}ou^e({\cal X})$  consisting of 1-dimensional complex subspaces of ${\cal X}$ which represent  the class $e$ and are contained in $p^{-1}(V)$.  

For every $D\in {\cal V}$ the projection $p(D)\subset V$ is a finite union of points. On the other hand it's easy to see that any small flat deformation of a tree of smooth rational curves is again a tree of smooth rational curves, in particular it is connected. Therefore, replacing ${\cal V}$ by a smaller neighborhood if necessary, the conclusion (\ref{1}) will be fulfilled.\\

The normal bundle ${\cal N}_{\cal X}(D_0)$ of $D_0$ in ${\cal X}$  fits in a short exact sequence on $D_0$:
$$0\to {\cal N}_{X_{b_0}}(D_0)={\cal O}_{D_0}(D_0) \to {\cal N}_{\cal X}(D_0) \to D_0\times T_{b_0}(B)\to 0\ .
$$
Since $h^1({\cal O}_{D_0}(D_0))=h^1({\cal O}_{D_0})=0$  we obtain 
$$H^0({\cal N}_{\cal X}(D_0))=T_{b_0}(B)\ ,\ H^1({\cal N}_{\cal X}(D_0))=\{0\}\ .$$
 This shows that ${\cal D}ou^e({\cal X})$ is smooth at 
$D_0$ and the projection on $B$ defined by $\mathfrak{p}$ is a local biholomorphism at $D_0$. This proves (\ref{2}).  Let ${\cal U}\subset{\cal V}$ be an open neighborhood of $D_0$ on which $\mathfrak{p}$ is biholomorphic, and put $U:=\mathfrak{p}({\cal U})\subset B$. We obtain a holomorphic family of trees of rational curves $(D_b)_{b\in U}$ with $D_{b_0}=D_0$.\\

The third statement can be proved by induction with respect to the number $m$ of irreducible components of $D_0$. Consider an irreducible component $D_m$ of $D_0$ which is an exceptional curve of the first kind. $D_m$ is itself a simply exceptional divisor, so the statements (\ref{1}),   (\ref{2})  apply to $D_m$, and, taking ${\cal U}$ and $U$ sufficiently small, we obtain a holomorphic family 
$$(D_{m,b})_{b\in U}\ ,\ D_{m,b}\subset X_b$$
 of simply exceptional divisors.  But, by the stability theorem of Kodaira,   any small deformation of an exceptional curve of the first kind is again an exceptional curve of the first kind, so we may assume that all $D_{m,b}$ are all exceptional curves of the first kind.  This proves the statement when $m=1$. Suppose $m>1$. Contracting in $X_b$ the curve  $D_{m,b}$ to a point $y_b$ for every $b\in B$, we obtain a commutative diagram
$$
\begin{array}{lcc}
{\cal X}_U\hspace{-5mm}&\stackrel{\pi}{\longrightarrow}&\hspace{-5mm} {\cal Y}\\
&p_U\ \hspace{-2mm}\searrow  \ \  \ \ \swarrow \hspace{-1mm}s&\\
&\hspace{3mm}U&
\end{array}
$$
where $\pi$ is proper and bimeromorphic whose restriction to any fiber $X_b$ is just the blowing down map  $\pi_b:X_b\to Y_b$   which contracts  $D_{m,b}\subset X_b$  to   $y_b$. Let $\mathcal{D}$, $\mathcal{D}_m\subset \mathcal{X}_U$ be the classifying divisors  of the families $(D_b)_{b\in U}$, $(D_{m,b})_{b\in U}$. Supposing $U$ connected, we see that $\mathcal{D}_m$ is irreducible. One has $(\mathcal{D}\cap \mathcal{D}_m)\cap X_{b_0}=D_{m,b_0}$, so (since we supposed $m>1$) one has ${\cal D}\ne {\cal D}_m$.
Let ${\cal D}'$ be the divisor which coincides with ${\cal D}$ if  ${\cal D}_m\not \subset \mathcal{D}$ and is obtained from $\mathcal{D}$ by removing the ${\cal D}_m$ if ${\cal D}_m\subset \mathcal{D}$. The projection on ${\cal Y}$ of every irreducible component of ${\cal D}'$ is   reduced and irreducible of codimension 1, so ${\cal D}''=\pi({\cal D}')$ is an effective divisor of ${\cal Y}$. Any fiber  $D_{b}''$ of this divisor over a point $b\in U$ is an effective   divisor of $Y_b$ (more precisely  a tree of rational curves in $Y_b$) so  $\mathcal{D}''$ is flat over $U$ by Lemma \ref{flatnesslemma} below.  By induction assumption all 
 fibers $D''_b$ are simply exceptional divisors.    Note now that $D_b$ is the pre-image   $\pi_b^*(D''_b)$  of $D''_b$ in $X_b$.
Therefore $D_b$ will be simply exceptional, too.
\end{proof}
\begin{Lem} \label{flatnesslemma} Let $p:{\cal X}\to B$ be a submersion of  complex manifolds with  $\dim({\cal X})=n$, $\dim(B)=m$. Let ${\cal D}\subset {\cal X}$ be an effective divisor such that ${\cal D}\cap X_b$ is a (possibly empty) effective divisor of the fiber $X_b$ for every $b\in B$. Then ${\cal D}$ is flat over $B$.
\end{Lem}
\begin{proof} Let $x\in {\cal D}$ and let $\phi:U\to\C$ be a local equation for ${\cal D}$ around $x$, i.e. $U$ is an open  neighborhood of $x$ in ${\cal X}$ and ${\cal D}\cap U=Z(\phi)$. Consider the map $\psi:=(p,\phi):U\to B\times\C$, and note that the fiber $\psi^{-1}(\psi(x))$ of this map can be identified with ${\cal D}\cap U\cap X_{p(x)}$ so the dimension at $x$ of this fiber has minimal value $n-m-1$. By the fiber dimension semicontinuity theorem we see that, replacing $U$ with a smaller open neighborhood if necessary, we may assume that 
$\dim_{\xi}\ \psi^{-1}(\psi(\xi))=n-m-1$ for every $\xi\in U$. Therefore $\psi$ will be flat by the flatness criterion given by the Corollary on p. 158 in \cite{Fi}.  Now we apply the base change property of the flat morphisms to the base change $B\ni b\mapsto (b,0)\in B\times\C$ and we get that ${\cal D}\cap U$ is flat over $B$.
\end{proof}
For a    holomorphic family ${\cal X}\to B$ of compact complex surfaces we will  denote by $k=(k_b)_{b\in B}\in H^0(B,\underline{H}^2)$ the section defined by the classes 
$k_b:=c_1({\cal K}_{X_b})$, $b\in B$.
\begin{Cor} Let ${\cal X}\to B$ be a    holomorphic family of compact complex surfaces, and let $e=(e_b)_{b\in B}\in H^0(B,\underline{H}_2)$ such that $\langle e,k\rangle=-1$ and $e\cdot e=-1$ with respect to the intersection form induced from any $H_2(X_b,\Z)$. Then the subset
$$B_e:=\{b\in B|\ X_b \hbox{ admits a simply exceptional divisor representing the class }e_b  \}\subset B
$$
is open with respect to the classical topology.
\end{Cor}

\begin{Lem} \label{NumExcept} Let $S$ be a  class VII surface and let $C$ be an effective divisor in $X$  with 
$$  C^2 = (C,  {\cal K}_S)=-1\ .$$
 Then $C=E+ \G$, where $E$ is a simply exceptional divisor, $p\ge 0$, and $\G=f^\star \G'$  is the pre-image of a homologically trivial effective divisor $\G'\subset S'$  via the contraction   $f:S\to S'$  to the minimal model.
 \end{Lem}
\begin{proof} If the minimal model of $S$ is an Inoue surface $S'$ with $b_2(S')=0$, the result is clear, because an Inoue surface with $b_2=0$ does not contain any curve. Therefore we shall suppose that the minimal model is a Hopf surface or a minimal class VII surface with $b_2>0$.   Let $C=\sum_{i=1}^m n_i C_i$ ($n_i>0$) be the  decomposition of $C$ as linear combination of irreducible components. 
We prove the result by induction on the sum $\sigma_C:=\sum_{i=1}^m n_i\ge 1$. 
If $\sigma_C=1$, then $C=C_1$ is irreducible with  $C_1^2=(C,\mathcal{K}_S)=-1$, so it is an exceptional curve of the first kind  by Proposition 2.2 in \cite{BHPV}. 

Suppose now $\sigma_C\ge 2$. We have
$$-1=(C,{\cal K}_S)=\sum_i n_i (C_i,{\cal K}_S)$$
hence there exists an index $j$ such that $(C_j,{\cal K}_S)<0$. Since the intersection form of a class VII surface is negative definite, one has $C_j^2\leq 0$, with equality if and only if $C_j$ is homologically trivial. This would contradict $(C_j,{\cal K}_S)<0$, so necessary $C_j^2<0$.  By Proposition 2.2 in \cite{BHPV} again it follows that $C_j$ is an exceptional curve of the first kind.
\vspace{2mm}\\
{\it Case 1.} The homology classes $[C]_\mathbb{Q}$,  $[C_j]_\mathbb{Q}$ coincide in $H_2(S,\mathbb{Q})$. In this case 
 the effective divisor $\G:=C-C_j$ is $\mathbb{Q}$-homologically trivial. Therefore the effective divisor  $\G':=f_*(\G)\subset S'$ is $\mathbb{Q}$-homologically trivial, too. We can write $\G=f^*(\G')+F$, where $F$ is an effective divisor contained in the exceptional divisor of the contraction map $f$. But $F$ must be also $\mathbb{Q}$-homologically trivial, so it is empty. This shows that $\G=f^*(\G')$, where $\G'$ is a $\mathbb{Q}$-homologically trivial effective divisor  in the minimal model $S'$.  But, using Enoki's theorem concerning the classification of class VII surfaces admitting a numerically trivial divisor, it follows that on such a surface any $\mathbb{Q}$-homologically trivial effective divisor is homologically trivial.
\vspace{2mm}\\
{\it Case 2.} The homology classes $[C]_\mathbb{Q}$, $[C_j]_\mathbb{Q}$ are different. We contract the exceptional curve $C_j$ obtaining a blowing down map $g:S\to S_0$. The homology group $H_2(S,\Z)$ has a decomposition 
$$H_2(S,\Z)=\Z[C_j]\oplus H^\bot\ ,
$$
which is orthogonal with respect to the intersection form $q_S$ of $S$, and $g_*$ induces an isomorphism
\begin{equation}\label{iso}(H^\bot, \resto{q_S}{H^\bot\times H^\bot})\textmap{\simeq \resto{g_*}{H^\bot}}  (H_2(S_0,\Z), q_{S_0})
\end{equation}
 
Since the rational intersection form $q_S^{\R}$ is negative definite, we have
\begin{equation}\label{ineq}0\leq| C C_j|= \left |q_S^\R([C]_\R, [C_j]_\R)\right|\leq \frac{1}{2}(-C^2-C_j^2)=1\ ,
\end{equation}
with equality if and only if $[C]_\R, [C_j]_\R$ are colinear. But since these classes belong to the lattice $\mathrm{im}(H_2(S,\Z)\to H_2(S,\R))$  and $(C_j,{\cal K}_S)=  (C,{\cal K}_S)=-1$, the colinearity condition implies $ [C]_\R=[C_j]_\R$, which cannot hold in our case. Therefore the inequality (\ref{ineq}) implies $C C_j=0$, i.e. $[C]\in H^\bot$. Using the isomorphism (\ref{iso}) we see that
$g_*(C)^2=C^2=-1$. On the other hand ${\cal K}_{S}=g^*({\cal K}_{S_0})(C_j)$, so
$$-1=(C,{\cal K}_{S})=(C,g^*({\cal K}_{S_0})(C_j))=CC_j+(C,g^*({\cal K}_{S_0}))=(g_*(C), {\cal K}_{S_0})\ .
$$
This shows that the effective divisor $C_0:=g_*(C)$ of $S_0$ satisfies the conditions
$$C_0^2=(C_0, {\cal K}_{S_0})=-1\ .
$$
Since $\sigma_{C_0}<\sigma_C$ we can apply the induction assumption to $C_0$. Note   that $C=g^*(C_0)$. To see this it suffices to see that  
\begin{itemize}
\item[-] they differ by a multiple  of $C_j$,
\item[-] the corresponding homology classes coincide, because they belong to $H^\bot$ and have the same image via $g_*$.
\end{itemize} 

By induction assumption we know that $C_0$ decomposes as $C_0=E_0+f_0^*(\G_0)$, where $f_0:S_0\to S'$ is the contraction to the minimal model of $S_0$ (which is also the minimal model of $S$), and $\G_0$ is a homologically trivial divisor in $S_0$. Using the equality $C=g^*(C_0)$,   we get $C=g^*(E_0)+g^*(f_0^*(\G_0))$. It suffices to note that $g^*(E_0)$ is a simply exceptional divisor.
\end{proof}

The minimal class VII surfaces admitting a (non-empty) homologically trivial   effective divisor $\G$ are classified \cite{E}, we can describe explicitly the second term $f^*(\G)$ in the decomposition given by  Lemma \ref{NumExcept}.

\begin{Rem} If $\G$ is a (non-empty) homologically trivial   effective divisor on  a  minimal class VII surface $X$, then the pair $(X,\G)$ is one of the following:
\begin{enumerate}
\item  $X$ is an elliptic Hopf surface and $\G=\sum_{i=1}^k n_i C_i$, where $n_i\geq 0$, $\sum_{i=1}^k n_i>0$ and  $C_i$ are elliptic curves,
\item $X$ is a Hopf surface with two elliptic curves $C_1$, $C_2$, and $\G=n_1 C_1+n_2 C_2$  with $n_i\geq 0$, $n_1+n_2>0$,
\item $X$ is a  Hopf surface with an elliptic curve  $C$  and $\G=n C$ with $n> 0$,
\item $X$ is an Enoki surface  and $\G=n C$, where $n>0$ and $C\subset X$ is a cycle of rational curves.
\end{enumerate}
\end{Rem}

\begin{Lem}\label{closed}  Let ${\cal X}\to B$ be a holomorphic family of class VII surfaces, let $e=(e_b)_{b\in B}\in H^0(B,\underline{H}_2)$, $b_0\in B$ and $(b_n)_{n\in\mathbb{N}^*}$ a sequence in $B$ converging to $b_0$. Consider, for every $n\in\mathbb{N}^*$, a simply  exceptional divisor $D_n\subset X_{b_n}$ in the class $e_{b_n}\in H_2(X_{b_n},\Z)$ such that the sequence $(D_n)_{n\in\mathbb{N}^*}$  converges in ${\cal B}^e_B({\cal X})$ to a 1-cycle $D_0\subset X_{b_0}$.  Then
\begin{enumerate}
\item The limit cycle $D_0$ is a simply exceptional divisor in $X_{b_0}$.
\item Regarding $D_n$, $D_0$ as  complex subspaces of ${\cal X}$, one has 
$\lim_{n\to\infty} D_n=D_0$ in the Douady space ${\cal D}ou({\cal X})$.
\end{enumerate}
\end{Lem}
\begin{proof} By Lemma 
\ref{NumExcept} we see that $D_0$ decomposes as $D_0=E_0+f_0^*(\G_0)$, where $E_0$ is a simply exceptional divisor in the class $e$,   $f_0:X_{b_0}\to X'_{b_0}$ is the contraction to the minimal model, and $\G_0$ is a homologically trivial effective divisor in $X'_{b_0}$. By Proposition \ref{open} there exists an open neighborhood ${\cal U}$ of $E_0$ in ${\cal D}ou({\cal X})$ such that every element $E\in {\cal E}$ is a simply exceptional divisor in the fiber $X_{\mathfrak{p}(D)}$ where $\mathfrak{p}:{\cal U}\to U$ is a biholomorphism  on an open neighborhood $U$ of $b_0$. For sufficiently large $n$ one has $b_n\in  U$ we obtain  a simply exceptional divisor  $E_n:=\mathfrak{p}^{-1}(b_n)$ in $X_{b_n}$, and obviously
$$\lim_{n\to\infty} E_n=E_0
$$
in the Douady space ${\cal D}ou({\cal X})$, so also in the Barlet space ${\cal B}^e_B({\cal X})$.  But, for every $n\in\mathbb{N}^*$, $D_n$ and $E_n$ are simply exceptional divisors representing the same homology class, so by Remark \ref{unicity-div} we have $D_n=E_n$ for any $n\in\mathbb{N}^*$. But the Barlet cycle-space   is Hausdorff  so, since the limit of a convergent sequence in a Hausdorff space is unique, $D_0=E_0$. This  shows that $D_0$ is a simply exceptional divisor, and the that $\lim_{n\to\infty} D_n=D_0$ also holds in the Douady space. 
\end{proof}

Let $p:{\cal X}\to B$ be a   deformation of class VII surfaces parameterized by a  complex manifold $B$.   Let again $k\in  H^0(B,\underline{H}^2)$ be the element defined by the family of Chern classes $k_b:=c_1({\cal K}_{X_b})$, $b\in B$.  Using the exponential short exact sequence 
$$0\to \underline{\Z}\to{\cal O}\to{\cal O}^*\to 0
$$
on ${\cal X}$, the vanishing of the cohomology groups $H^2(X_b,{\cal O}_{X_b})$  we obtain the exact sequence of sheaves on $B$
\begin{equation}\label{ExSeq}
0 \to  \qmod{R^1p_*({\cal O}_{\cal X})}{\underline{H}^1}\to R^1p_*({\cal O}^*_{\cal X})\to \underline{H}^2\to 0\ .
\end{equation}
The quotient sheaf on the left (respectively the sheaf in the middle) is the sheaf of holomorphic sections in a locally trivial complex Lie group fiber bundle $\Pic^0$ (respectively $\Pic$) over $B$ whose fiber over $b\in B$ is $\Pic^0(X_b)$ (respectively $\Pic(X_b)$).

On a class VII surface the canonical morphism $H^1(X,\C)\to H^1(X,{\cal O})$ is an isomorphism, so  $R^1p_*({\cal O}_{\cal X})$ can be identified with $\underline{H}^1\otimes{\cal O}_B$.  It is easy to see that the sheaf $\underline{H}^1$ (which is obviously locally constant with fiber isomorphic to $\Z$) is in fact constant; it suffices to see that the degree maps associated with a smooth family of Gauduchon metrics on the fibers define    canonical orientations of the real lines $H^1(X_b,\R_{>0})\subset  H^1(X_b,\C^*)$, so canonical orientations of the lines $H^1(X_b,\R)$, so canonical generators of the infinite cyclic groups $H^1(X_b,\Z)$. Therefore the complex Lie group bundle $\Pic^0$ is in fact trivial with fiber $\C^*$. It is convenient to choose a trivialization $\phi=(\phi_b:\C^*\to\Pic^0(X_b))_{b\in B}$  of this fiber bundle such that for every $b\in B$ one has 
$$\lim_{z\to 0}\deg_g(\phi_b(z))=
\infty$$
with respect to a Gauduchon metric on a fiber $X_b$.
 
For every section $c\in H^0(B,\underline{H}^2)$ we obtain a subbundle $\Pic^c$  of $\Pic$, which  becomes a principal $\C^*$-bundle via the chosen trivialization $\phi:B\times\C^*\to\Pic^0$.  The associated line bundle ${\cal P}^c$  on $B$ can be obtained by adding formally  to $\Pic^c$ a zero section $\{0_b|\ b\in B\}$, such that for every $b\in B$ one has
$$\forall\lambda\in \Pic^{c_b}(X_b)\  \ \lim_{z\to 0} \phi_b(z)\otimes \lambda=0_b \ .
$$

\begin{Th}\label{Zar} Let $e=(e_b)_{b\in B}\in H^0(B,\underline{H}_2)$ such that $e^2=\langle e,k\rangle=-1$, and let $B_e$ be the open subset of $B$ consisting of points $b$ for which $X_b$ admits a simply exceptional divisor $E_b$ in the class $e_b$. 
If $B_e\ne\emptyset$ then
\begin{enumerate}
\item There exists an effective divisor ${\cal E}_0\subset p^{-1}(B_e)$ flat over $B_e$  whose fiber over a point $b\in B_e$ is $E_b$.
\item The map $\sigma_e:B_e\to {\cal P}^e$ defined by 
$b\mapsto[{\cal O}_{X_b}(E_b)]$ extends to a holomorphic section $s_e$ of the line bundle ${\cal P}^e$ over $B$.
\item One has $B_e=B\setminus Z(s_e)$, where $Z(\cdot)$ stands for the vanishing locus of a section. In particular, the complement of $B_e$ is a divisor $H_e:=Z(s_e)$, so $B_e$ is Zariski open.
\end{enumerate}
\end{Th}

\begin{proof} Using Remark \ref{unicity-div} and Proposition \ref{open} we see that the map 
$$B_e\ni b\mapsto E_b\in{\cal D}ou({\cal X})$$
 is well-defined and holomorphic. The first statement follows now from the universal property of the Douady space. 

Let $(g_b)_{b\in B}$ be a smooth family of Gauduchon metrics on the fibers $X_b$ normed such that for every $b\in B$ it holds $\deg_{g_b}(\phi_b(z))=-\log|z|$. The map ${\cal P}^e\to[0,\infty)$ given by
$$\Pic^e(X_b)\ni[{\cal L}]\mapsto   e^{-\deg_{g_b}({\cal L})}\ ,\ 0_b\mapsto 0$$
defines a norm $\|\cdot\|$ on the holomorphic line bundle ${\cal P}^e$ on $B$.  We claim
\vspace{2mm}\\
{\it Claim:  For every $b_0\in \bar B_e\setminus B_e$ one has $\lim_{b\to b_0} \| \sigma_e(b)\|=0$.}
\vspace{2mm}\\
Indeed, if not, there exists a  sequence $(b_n)_{n\in\mathbb{N}^*}$ in $B_e$ converging to $b_0$ and $\varepsilon>0$  such that $\| \sigma_e(b_n)\|\geq \varepsilon$ for every $n\in \mathbb{N}^*$. This implies that the sequence 
$$(\deg_{g_{b_n}}(\sigma_e(b_n )))_{n\in \mathbb{N}^*}=(\mathrm{vol}_{g_{b_n}}(E_{b_n} ))_{n\in \mathbb{N}^*}$$
is bounded, so there exists a subsequence of $(E_{b_n} )_{n\in \mathbb{N}^*}$ converging in the Barlet space ${\cal B}_B^e({\cal X})$ \cite{Ba}. By Lemma \ref{closed} the limit divisor $E_0$ in $X_{b_0}$ will be simply exceptional, so $b_0\in B_e$, which contradicts the choice of $b_0$. This proves the claim.
\vspace{2mm}\\

We define now a section $s_e: B\to {\cal P}^e$ by
$$s_e(b):=\left\{ \begin{array}{ccc}\sigma_e(b)&\rm if&b\in B_e\\
0_b&\rm if & b\not\in B_e \end{array}\right.\ .
$$
Using the claim proved above we see that $s_e$ is continuous. The first statement  of the theorem follows now  from Rado's theorem, and the second follows from the first and the explicit construction of the section $s_e$.
\end{proof}
\begin{Rem}  The functor which associates to a pair $(p:{\cal X}\to B,e)$ consisting of a holomorphic family of class VII surfaces and a section $e=(e_b)_{b\in B}\in H^0(B,\underline{H}_2)$ with 
$$e^2=\langle e, k\rangle =-1$$
 the (possibly empty) effective divisor $H_e\subset B$, commutes with base change.
Therefore this functor should be interpreted as an effective Cartier divisor in the moduli stack classifying pairs $(X,\eta)$ consisting of a class VII surface $X$ and a class $\eta\in H_2(X,\Z)$ with $\eta^2=\langle \eta, {\cal K}_X\rangle =-1$.
\end{Rem} 
\begin{Rem} The underlying codimension 1 analytic set $\mathfrak{H}_e\subset B$ associated with $H_e$ coincides with the hypersurface defined by Dloussky in \cite{D4}.
\end{Rem}

\section{Extension theorems}

Let $p:{\cal X}\to B$ be a holomorphic family of class VII surfaces. Using the notations and the definitions introduced in   the previous section, fix a section  $e=(e_b)_{b\in B}\in H^0(B,\underline{H}_2)$ such that $e^2=\langle e,k\rangle =-1$, and let $B_e:=B\setminus H_e\subset B$ the associated Zariski open subset (see Theorem \ref{Zar}).  We know that $H_e$ is the vanishing divisor of a holomorphic section $s_e\in H^0(B,{\cal P}^e)$, which on $B_e$ is given by $b\mapsto [{\cal O}_{X_b}(E_b)]$, $E_b$ being the unique simply exceptional divisor on $X_b$ representing the class $e_b$.

As we explained in the introduction the fundamental problem studied in this article is the evolution of the divisor $E_b$ as $b$ tends to a point $b_0\in H_e$. Since   the volume of $E_b$ with respect to a smooth family of Gauduchon metrics on the fibers tends to infinity as $b\to b_0$, we cannot expect to understand this evolution using standard complex geometric tools applied to the given family.

In this section we will prove that, under certain assumptions, a lift $\widetilde E_b$   of $E_b$ in the universal cover $\widetilde X_b$ of $X_b$ does have a limit as $b\to b_0$, and this limit is a non-compact divisor of $\widetilde X_{b_0}$. Moreover, the family $(\widetilde E_b)_{b\in B}$ extends to a flat family of divisors over the whole base $B$. We can prove this type of extensions results under two types of assumptions. First we will assume that all the fibers $X_b$ are class VII surfaces with the topological type of class VII surfaces having a non-separating strictly pseudo-convex embedded 3-sphere; second we will suppose that the subset $B_0\subset B$ consisting of points $b$ for which $X_b$ is an unknown class VII surface  is contained in an analytic subset of codimension $\geq 2$.\\

The diffeomorphism type of a GSS surface  with $b_2=n$ is $(S^1\times S^3)\# n\bar{\mathbb{P}}^2$, so if the GSS conjecture was true this would be the diffeomorphism type of {\it any} minimal class VII surface with $b_2=n>0$.  Moreover, by a result of Nakamura \cite{N1}, any minimal  class VII surface with $b_2=n>0$ containing a cycle of curves is a degeneration (``big deformation") of a family of blown up primary Hopf surfaces, so it also has  the  diffeomorphism type of a GSS surface.  By the results of \cite{Te2}, \cite{Te3} every minimal class VII surface with $b_2\in\{1,2\}$ does contain  a cycle, so by Nakamura's result, it has the same diffeomorphism type.  
Therefore, studying class VII surfaces   with the diffeomorphism (or, more generally,  homeomorphism) type $(S^1\times S^3)\# n\bar{\mathbb{P}}^2$, and studying  deformations of such surfaces, is an important problem.

In general, if $M$ is a topological 4-manifold $M$ with this topological type, then every 2-homology class $e\in H_2(M,\Z)$  is represented by an embedded surface with simply connected  components, in particular it lifts to  $H_2(\widetilde M,\Z)$. More precisely, $H_2(\widetilde M,\Z)$ is naturally an $n$-dimensional free  $\Z[F]$-module, where $F$ stands for the multiplicative group $\mathrm{Aut}_M(\widetilde M)=\{f^k|\ k\in\Z\}\simeq\Z$, and the quotient 
$$H_2(\widetilde M,\Z)\otimes_{\Z[F]}\Z=\qmod{H_2(\widetilde M,\Z)}{(\id-f)H_2(\widetilde M,\Z)}$$
is naturally isomorphic to $H_2(M,\Z)$. Here   $\Z$ is regarded as a $\Z[F]$-algebra via the trace morphism 
$$\sum_{n\in\Z} a_n f^n\mapsto\sum_n a_n\ ,$$
 and    $(\id-f)$ stands for the ideal $(\id-f)=(\id-f)\Z[F]\subset \Z[F]$ generated  by $\id-f$.

 Let  $p:{\cal M}\to B$ be a  locally trivial topological fiber bundle with $(S^1\times S^3)\# n\bar{\mathbb{P}}^2$ as fiber and  2-connected basis  $B$,  $\alpha:\widetilde{\cal M}\to{\cal M}$ the universal cover of the total space, and $\widetilde p:=p\circ\alpha:\widetilde{\cal M}\to B$ the induced fiber bundle. Since $B$ is 2-connected, the homotopy long exact sequence associated with the locally trivial fiber bundle $p$ shows that the fiber embedding $\iota_b:M_b\hookrightarrow {\cal M}$ induces an isomorphism $\pi_1(M_b,m)\textmap{\simeq} \pi_1({\cal M},m)$, so  the induced projection $\alpha_b:\widetilde M_b\to M_b$ is a universal cover of the fiber $M_b$ for every $m\in M_b$. 
Note also that the assumption ``$B$ is 2-connected" implies that the locally constant sheaves $\underline{H}_i$, $\underline{H}^i$ on $B$ are constant, for every $i\in\mathbb N$.  Denoting by $\underline{\widetilde H}_i$, $\underline{\widetilde H}^i$ the corresponding sheaves associated with the fiber bundle $\widetilde p$ we obtain sheaf epimorphisms $\underline{\widetilde H}_i\to \underline{ H}_i$ and sheaf monomorphisms $\underline{H}^i\to \underline{\widetilde H}^i$ defined by the morphism sequences $(\alpha_{b*})_{b\in B}$, $(\alpha_{b}^*)_{b\in B}$.

We denote by $  f$ a generator of $\mathrm{Aut}_{\cal M}(\widetilde {\cal M})$ and by $f_b$ the induced generator of $\mathrm{Aut}_{M_b}(\widetilde M_b)$.
 With these notations   we obtain, for any $b\in B$, a commutative diagram with horizontal short exact sequences:
$$
\begin{array}{ccccccc}
0\to &(\id-f_b) H_2(\widetilde M_b,\Z)&\hookrightarrow &H_2(\widetilde M_b,\Z)&\textmap{\alpha_{b,*} } &H_2(M_b,\Z)&\to 0\vspace{2mm}\\  
&\uparrow \simeq &&\uparrow\simeq  &&\uparrow\simeq \vspace{2mm}\\  
0\to &(\id-  f) H^0(B,\underline{\widetilde H}_2)&\hookrightarrow &H^0(B,\underline{\widetilde H}_2)&\textmap{  \alpha_{*} } &H^0(B,\underline{H}_2)&\to 0
\end{array}\ .
$$

Let $X$ be a class VII surface with  topological type $(S^1\times S^3)\# n\bar{\mathbb{P}}^2$ and $\Sigma$ a strictly pseudo-convex non-separating embedded 3-sphere in $X$. A lift $\widetilde\Sigma$ of $\Sigma$ to $\widetilde X$ separates $\widetilde X$ in two manifolds with common boundary $\widetilde\Sigma$, so we can define the pseudo-convex and pseudo-concave ends  of $\widetilde X$ as in the case of GSS surfaces (see section \ref{DefClassVII}):   the  pseudo-convex end of $\widetilde X$  is  the end defined by the connected component of $\widetilde X\setminus\widetilde\Sigma$ whose closure has strictly pseudo-concave boundary.  One can prove that  this labeling  of the ends is coherent, i.e.  independent of the choice of $\Sigma$ and $\widetilde\Sigma$.
\begin{Th}\label{first} (first extension theorem) Let $p:{\cal X}\to B$ be a holomorphic family of class VII surfaces with 2-connected basis $B$, such that every fiber $X_b$
\begin{enumerate}
\item has the topological type $(S^1\times S^3)\# n\bar{\mathbb{P}}^2$, 
\item admits  a non-separating strictly pseudo-convex embedded 3-sphere $\Sigma_b$.
\end{enumerate}  Let $e=(e_b)_{b\in B}\in H^0(B,\underbar{H}_2)$ with $e^2=\langle e, k\rangle=-1$ such that $B_e\ne\emptyset$.
Let $\alpha:\widetilde {\cal X}\to{\cal X}$ be a universal cover of ${\cal X}$,   and  $\widetilde e=(\widetilde e_b)_{b\in B}$  a lift of $e$ in $H^0(B,\widetilde{\underline{H}}_2$).
Then there exists an effective divisor $\widetilde {\cal E}\subset\widetilde{\cal X}$ with the following properties:
\begin{enumerate}
\item For any point $b\in B_e$, the   fiber  $\widetilde E_b:=\widetilde {\cal E}\cap\widetilde X_b$  is a lift representing  $\widetilde e_b$ of  the unique simply exceptional divisor $E_b$ in the class $e_b$,
\item For any point  $b\in H_e$, the   fiber  $\widetilde E_b$ is bounded towards the pseudo-convex  end of $\widetilde X_b$, but unbounded towards its    pseudo-concave end.
\item $\widetilde {\cal E}$ is flat over $B$.
\end{enumerate}
\end{Th}
\begin{proof} 
Using Theorem \ref{Zar} we get an effective divisor ${\cal E}_0 \subset p^{-1}(B_e)$, flat over $B_e$, whose fiber  over a point $b\in B_e$ is the unique simply exceptional divisor $E_b\subset X_b$ in the class $e_b$. Since $E_b$ is a tree of rational curves (so simply connected), it can be lifted to $\widetilde X_b$, and the set of lifts can be identified with the set $\alpha_{b,*}^{-1}(e_b)\subset H_2(\widetilde X_b,\Z)$. Moreover, since $E_b$ has a simply connected  neighborhood in ${\cal X}$,   such a lift  can be   locally chosen such that it depends holomorphically on $b\in B_e$. 

For $b\in B_e$ let $\widetilde E_b$ be the lift of $E_b$ representing the given class $\widetilde e_b\in H_2(\widetilde X_b,\Z)$. Using Proposition \ref{open} we see that map $B_e\ni b\mapsto  \widetilde E_b\in{\cal D}ou^e({\cal X})$ is biholomorphic. The union
$$\widetilde {\cal E}_0 :=\bigcup_{b\in B_e} \widetilde E_b
$$
is an effective divisor of $\widetilde p^{-1}(B_e)\subset\widetilde{\cal X}$, and the restriction $\resto{\alpha}{\widetilde{\cal E}_0}:\widetilde{\cal E}_0\to {\cal E}_0$ is biholomorphic.  \\

The idea of the proof is to prove, using the Remmert-Stein theorem, that  the closure $\widetilde{\cal E}:=\bar{\widetilde {\cal E}}_0$ of $\widetilde{\cal E}_0$ in $\widetilde{\cal X}$ is a divisor, and that this divisor has the desired properties. The irreducible components of   $\widetilde H_e:=\widetilde{\cal X}\setminus \widetilde p^{-1}(B_e)=\widetilde p^{-1}(H_e)$ have the form $\widetilde p^{-1}(H)$, where $H$ is an irreducible component of the divisor $H_e$ (see Theorem \ref{Zar}). Let $H_0$ be such a component and $b_0\in H_0$. Using the second property of the  fibers, we get an embedding $s_0:S^3\to X_{b_0}$ whose image $\Sigma_0$ is non-separating and strictly pseudo-convex. The image $\widetilde\Sigma_0:=\mathrm{im}(\widetilde s_0)$ of a lift $\widetilde s_0:S^3\to \widetilde X_{b_0}$   of this embedding to $\widetilde X_{b_0}$  separates $\widetilde X_{b_0}$ in two manifolds $\widetilde X_{b_0}^\pm$, the first with strictly pseudo-concave, the second with strictly pseudo-convex  boundary $\widetilde\Sigma_0$.

 We can deform the lift  $\widetilde s_0$ to get a smooth family of embeddings $(\widetilde s_b:S^3\to\widetilde{\cal X})_{b\in U}$, where $U$ is a connected open neighborhood of $b_0$ in $B$, such that  $\widetilde s_{b_0}=\widetilde s_0$, and for every $b\in U$   the image   $\widetilde\Sigma_b:=\mathrm{im}(\widetilde s_b)$ is a strictly pseudo-convex separating hypersurface of $\widetilde X_b$.
Therefore, for every $b\in U$ the  universal cover $\widetilde X_b$ of $X_b$ decomposes as the union of two manifolds $\widetilde X_b^\pm$, the first with pseudo-concave, the second with pseudo-convex boundary $\widetilde\Sigma_b$.
Similarly, $\widetilde{\cal X}_U:=\widetilde p^{-1}(U)$ decomposes as the union of two manifolds $\widetilde{\cal X}_U^\pm=\cup_{b\in U}\widetilde X_b^\pm$   with common boundary $\widetilde{\mathfrak{S}}_U:=\cup_{b\in U} \widetilde\Sigma_b$.   We choose the generator $ f$ of $\mathrm{Aut}_{\cal X}(\widetilde {\cal X})$ such that the induced automorphism $f_b$ of $\widetilde X_{b}$ maps $\widetilde X_{b}^-$ into itself (so it moves the points towards the pseudo-concave end) for every $b\in U$.

Choose a point $v_0\in B_e\cap U$. Since the divisor $\widetilde E_{v_0}\subset   \widetilde X_{v_0}$ is compact, we can find  $m\in\mathbb{N}$ sufficiently large such that $( f^{-m}_{v_0} (\widetilde X^+_{v_0}))\cap \widetilde E_{v_0}=\emptyset$. We claim  
\vspace{3mm}\\
{\it Claim:} $   f^{-m} (\widetilde {\cal X}^+_U)\cap\widetilde{\cal E}_0 =\emptyset$.
\\

Indeed, let $V\subset U\cap B_e$ the subset of points $v$ for which $(f^{-m}_{v} (\widetilde X^+_{v}))\cap \widetilde E_{v}=\emptyset$. Since all divisors $\widetilde E_b$, $b\in B_e$ are compact, it follows easily that $V$ is open in $U\cap B_e$. We will show that it is also closed in $U\cap B_e$; in order to check this,   let $(v_n)_{n\in\mathbb N^*}$ be a sequence in $V$ converging to a point $w\in U\cap B_e$. Since $\widetilde E_{v_n}\subset f^{-m}_{v_n}(\widetilde X^-_{v_n})$ for every $n\in\mathbb N^*$  it follows that   $\widetilde E_{w}\subset f^{-m}_w(\widetilde X^-_{w})$.  The intersection $\widetilde E_{w}\cap  f^{-m}_w(\widetilde\Sigma_w)$ is empty,  because $f^{-m}_w(\widetilde\Sigma_w)$ regarded as boundary of  $f^{-m}_w(\widetilde X^-_{w})$ is strictly pseudo-convex. Therefore $\widetilde E_{w}\cap f^{-m}_w(\widetilde X^+_{w})=\emptyset$, so $w\in V$.  

Therefore $V$ is both open and closed in   $U\cap B_e=U\setminus H_e$, which  is connected, so $V=U\cap B_e$. Therefore $(f^{-m}_u (\widetilde { X}^+_u))\cap\widetilde{\cal E}_0=\emptyset$ for every $u\in  U\cap B_e$, so for every $u\in U$; taking into account that 
$\bigcup_{u\in U}f^{-m}_u (\widetilde { X}^+_u)=  f^{-m} (\widetilde {\cal X}^+_U)$, 
this proves the claim.
\vspace{2mm}

Note now that $  f^{-m} (\widetilde {\cal X}^+_U)$ is a neighborhood of any point in $f^{-m-1}_{b_0} (\widetilde X^+_{b_0})$. Therefore the claim implies that    $\widetilde H_0$  contains points  which do not belong to  the closure $\overline{\widetilde {\cal E}}_0$ of the divisor $\widetilde{\cal E}_0\subset \widetilde {\cal X}\setminus \widetilde H_e$. Since this holds for every irreducible component $\widetilde H_0$ of $\widetilde H_e$, the existence of a divisor ${\cal E}$ satisfying the  first   property  stated in the theorem follows now from the Remmert-Stein theorem (see \cite{RS} Korollar zu Satz 12, p. 300). For the second  property note first that for any $b_0\in H_e$ we have $(f^{-m}_{b_0} (\widetilde X^+_{b_0}))\cap \widetilde E_{b_0}=\emptyset$, so $\widetilde E_{b_0}$ is indeed bounded towards the pseudo-convex end of $\widetilde X_{b_0}$. On the other hand  $\widetilde E_{b_0}$ cannot be compact  because, if it were, its projection on $X_{b_0}$ would be a limit of simply exceptional divisors, so itself simply exceptional (contradicting $b_0\in H_e$).
\vspace{2mm}

For the third property note that for every $b\in B$ the fiber $\widetilde E_b$ of $\widetilde{\cal E}$ is a divisor of $\widetilde X_b$, so the result follows from Lemma \ref{flatnesslemma}.

\end{proof}
\begin{Rem}\label{BM} Note that for $b\in B_e$ the divisor $\widetilde E_b$ represents the class $\widetilde e_b\in H_2(\widetilde X_b,\Z)$. This property cannot be extended for  $b\in H_e$ because in this case $\widetilde E_b$ is no longer compact. However, using the flatness property of $\widetilde {\cal E}$ over $B$ we see that for every $b\in B$
\begin{enumerate}
\item $c_1({\cal O}_{X_b}(\widetilde E_b))$ is the image of the Poincaré dual $PD(\widetilde e_b)\in H^2_c(\widetilde X_b,\Z)$ of $\widetilde e_b$ via the natural morphism   $H^2_c(\widetilde X_b,\Z)\to H^2 (\widetilde X_b,\Z)$.
\item $\widetilde E_b$ represents the image of $\widetilde e_b$ in the Borel-Moore homology group $H_2^{\rm BM}(\widetilde X_b,\Z)$.
\end{enumerate}
\end{Rem}
Taking into account Remark  \ref{BM} and Theorem \ref{first} we obtain immediately Theorem \ref{firstTh} stated in the introduction.
\begin{Rem}\label{withGSS} The assignment $(p:{\cal X}\to B,\alpha:\widetilde{\cal X}\to {\cal X},\widetilde e)\mapsto\widetilde {\cal E}\subset \widetilde{\cal X}$ given by the proof of Theorem \ref{first} is compatible with base changes $\chi:B'\to B$ satisfying $\chi^{-1}(B_e)\ne\emptyset$. 

Using this remark one can prove that the assumption $B_e\ne\emptyset$ in Theorem \ref{first} can be replaced with one of the following equivalent assumptions:
\begin{itemize}
\item[(i)] For every $b\in B$  there exists small deformations of $X_b$ admitting a simply exceptional divisors in the class $e_b$.
\item[(ii)] For every $b\in B$ the versal deformation of $X_b$  contains surfaces admitting a simply exceptional divisors in the class $e_b$.
\end{itemize} 
In particular, the theorem holds when all the fibers are GSS surfaces, because any GSS surface can be deformed into simply blown up primary Hopf surfaces. We will see in the next section that for families of GSS surfaces, the divisors $\widetilde E_b$, $b\in H_e$ are always given by infinite series of compact curves.
\end{Rem}
Using  Remark  \ref{withGSS} we obtain Theorem \ref{secondTh} stated in the introduction.
\\

Coming back to the conditions and the proof  of Theorem \ref{first}, consider the semigroup  $G_+:=\{ f^{n}|\ n\in\mathbb{N}\}$ of transformations of $\widetilde {\cal X}$ which act fiberwise with respect to the fibration $\widetilde p:\widetilde {\cal X}\to B$, where the generator $ f$ of $\mathrm{Aut}_{\cal X}(\widetilde{\cal X})$ is chosen again such that it moves the points towards the pseudo-concave ends of the fibers $\widetilde X_b$.  
\begin{Rem} \label{newrem}

 The closure $\overline{\widetilde {\cal F}_0}$ of  the effective divisor 
$$\widetilde {\cal F}_0:=G_+(\widetilde {\cal E}_0)=\bigcup_{n\in\mathbb{N}}  f^n(\widetilde{\cal E}_0)\subset \widetilde{p}^{-1}(B_e)$$
 in $\widetilde{\cal X}$ is also an effective divisor of $\widetilde{\cal X}$, flat  over $B$. 
\end{Rem} 
\begin{proof} We use the same arguments as in the proof Theorem \ref{first} and we take into account that $G_+$ moves $\widetilde{{\cal E}}_0$ towards the pseudo-concave end. More precisely,  the property $f^{-m} (\widetilde {\cal X}^+_U)\cap\widetilde{\cal E}_0 =\emptyset$  established in the proof implies $f^{-m} (\widetilde {\cal X}^+_U)\cap\widetilde{\cal F}_0 =\emptyset$. Indeed, if for a non-negative integer $n$ one had $x\in f^{-m} (\widetilde {\cal X}^+_U)\cap f^n(\widetilde{{\cal E}}_0)$, then 
$$f^{-n}(x)\in f^{-m-n} (\widetilde {\cal X}^+_U)\cap \widetilde{\cal E}_0\subset f^{-m} (\widetilde {\cal X}^+_U)\cap \widetilde{\cal E}_0\ ,$$
so the intersection $f^{-m} (\widetilde {\cal X}^+_U)\cap\widetilde{\cal E}_0 $ would be non-empty. Here we  used  the obvious inclusion $f^{-n}(\widetilde{\cal X}^+_U)\subset \widetilde{\cal X}^+_U$. 
\end{proof} 

Note that the similar statement formulated for the semigroup    $G_-:=\{ f^{-n}|\ n\in\mathbb{N}\}$ is not true in general.
\\

Our  second extension theorem concerns  families of class VII surfaces for which  the existence of strictly pseudo-convex embedded 3-spheres is assumed only for the fibers $X_b$, $b\in B\setminus A$, where $A$ is an analytic subset  of codimension $\geq 2$ of $B$ and  contained in the divisor $H_e$.  This theorem contains obviously  Theorem \ref{thirdTh} stated in the introduction as a special case. 

The ends of the  universal covers  $\widetilde X_b$ of the fibers form a  trivial double cover of $B$, so one can speak about the  pseudo-convex  and the pseudo-concave end of $X_b$ even for points $b\in A$.   
 
\begin{Th}\label{second} (second extension theorem) Let $p:{\cal X}\to B$ be a holomorphic family of class VII surfaces with 2-connected basis $B$.  Let $e=(e_b)_{b\in B}\in H^0(B,\underbar{H}_2)=H_2({\cal X},\Z)$ with $e^2=\langle e, k\rangle=-1$ such that $B_e\ne\emptyset$, and $A$   an analytic subset  of codimension $\geq 2$ of $B$ such that $A\subset H_e$ and, for every $b\in H_e\setminus A$, the fiber $X_b$ has the properties
\begin{enumerate}
\item has the topological type $(S^1\times S^3)\# n\bar{\mathbb{P}}^2$, 
\item admits  a non-separating strictly pseudo-convex embedded 3-sphere $\Sigma_b$.
\end{enumerate}
 Let $\alpha:\widetilde {\cal X}\to{\cal X}$ be a universal cover of ${\cal X}$,   and  $\widetilde e=(\widetilde e_b)_{b\in B}$  a lift of $e$ in $H^0(B,\widetilde{\underline{H}}_2$).
Then there exists an effective divisor $\widetilde {\cal E}\subset\widetilde{\cal X}$ with the following properties:
\begin{enumerate}
\item [(i)] For any point $b\in B_e$, the   fiber  $\widetilde E_b:=\widetilde {\cal E}\cap\widetilde X_b$  is a lift  of  the unique simply exceptional divisor $E_b$ in the class $e_b$. This lift represents   the class $\widetilde e_b$.
\item [(ii)] for any point  $b\in H_e=B\setminus B_e$, the   fiber  $\widetilde E_b$ is a divisor of $\widetilde X_b$ which is bounded towards the pseudo-convex  end of $\widetilde X_b$, but unbounded towards its    pseudo-concave end.
\item [(iii)] $\widetilde {\cal E}$ is flat over $B$.
\end{enumerate}
\end{Th}
\begin{proof}   Using Theorem \ref{Zar} again we obtain a divisor ${\cal E}_0\subset p^{-1}(B_e)$, flat over $B_e$, whose fiber  over a point $b\in B_e$ is the unique simply exceptional divisor $E_b\subset X_b$ in the class $e_b$. Applying our first extension theorem Theorem \ref{first} to the family $\widetilde p^{-1}(B\setminus A)\to B\setminus A$ we get an effective divisor $\widetilde{\cal E}^0\subset \widetilde p^{-1}(B\setminus A)$ flat over $B\setminus A$ satisfying the property (i) and also the property (ii)    for any point $b\in B\setminus A$.   So $\widetilde{\cal E}^0$ is the closure in $\widetilde p^{-1}(B\setminus A)$ of the lift $\widetilde {\cal E}_0\subset \widetilde p^{-1}(B\setminus H_e)$ of ${\cal E}_0$ obtained by lifting fiberwise the simply exceptional divisors $E_b$ in the homology class $\widetilde e_b$ (for $b\in B_e$).

Since $\widetilde p^{-1}(A)$ has codimension $\geq 2$ in $\widetilde{\cal X}$, the closure 
$$\widetilde {\cal E}:=\overline{\widetilde{\cal E}^0}=\overline{\widetilde{\cal E}_0}$$
 of $\widetilde {\cal E}^0$ in $\widetilde{\cal X}$ is an effective divisor  of $\widetilde{\cal X}$ by  the second Remmert-Stein extension theorem  (see \cite{RS} Satz 13 p. 299), where   $\overline{\phantom{I}}$ denotes everywhere closure in $\widetilde {\cal X}$. With this choice the first statement of the theorem is proved.  \\

For the second statement, the first problem is to show that for $b\in A$ the intersection $\widetilde E_b:= \widetilde {\cal E} \cap \widetilde X_b$ is a divisor of 
$\widetilde X_b$, i.e. that $\widetilde E_b$ does not coincide with the whole surface $\widetilde X_b$.  We will show first that $\widetilde {\cal E} \cap \widetilde X_b$ is bounded towards the pseudo-convex end, which will imply that $\widetilde E_b\cap \widetilde X_b\ne \widetilde X_b$, so this intersection is indeed a divisor of $\widetilde X_b$ with the desired property.

Consider the semigroup  $G_+:=\{ f^{n}|\ n\in\mathbb{N}\}$ of transformations of $\widetilde {\cal X}$ which act fiberwise with respect to the fibration $\widetilde p:\widetilde {\cal X}\to B$. Here we choose the generator $ f$ of $\mathrm{Aut}_{\cal X}(\widetilde{\cal X})$ as in the proof of Theorem \ref{first}, so that it moves the points towards the pseudo-concave ends of the fibers $\widetilde X_b$. 

Using Remark \ref{newrem} we see that the closure $\widetilde {\cal F}^0$   of  the effective divisor 
$$\widetilde {\cal F}_0:=G_+(\widetilde {\cal E}_0)=\bigcup_{n\in\mathbb{N}}  f^n(\widetilde{\cal E}_0)\subset \widetilde{p}^{-1}(B_e)$$
in $\widetilde p^{-1}(B\setminus A)$ is also an effective divisor of $\widetilde p^{-1}(B\setminus A)$, flat  over $B\setminus A$. 
By the second Remmert-Stein extension theorem  cited above we get an effective divisor   
$$   \widetilde{\cal F} =\overline{\widetilde{\cal F}^0}=\overline{\widetilde{\cal F}_0}  $$
  in $\widetilde {\cal X}$, where again $\overline{\phantom{I}}$ denotes everywhere closure in $\widetilde {\cal X}$. For every $n\in\mathbb{N}$, we have $  f^n(\widetilde {\cal E}_0)\subset \widetilde{\cal F} $  so, since the right hand space is closed,  we get  $  f^n( \widetilde {\cal E})\subset \widetilde{\cal F}$, which proves the inclusion
\begin{equation}\label{inclusion}
\forall n\in\mathbb{N}\ ,   f^n( \widetilde {\cal E})\subset \widetilde{\cal F} \  .
\end{equation}

Recall that the effective divisor $\widetilde{\cal E}_0$ of $\widetilde p^{-1}(B_e)$ is identified with ${\cal E}_0$ via the covering map $\alpha$. Since the set of irreducible components of ${\cal E}_0$ is finite (bounded by the minimal number of irreducible components of  the simply exceptional divisors $E_b$ as $b$ varies in $B_e$), the same will hold for $\widetilde{\cal E}_0$. Let $\widetilde {\cal C}_0$ be an irreducible component of $\widetilde{\cal E}_0$. Using the two Remmert-Stein theorems cited  above, we see that the closure $\widetilde {\cal C}^0$ of $\widetilde {\cal C}_0$ in $\widetilde p^{-1}(B\setminus A)$ is an effective divisor in this manifold, and that $\widetilde{\cal C}:=\overline{\widetilde{\cal C}^0}=\overline{\widetilde{{\cal C}}_0}$ is an effective divisor of $\widetilde{\cal X}$. We will show that 
\vspace{3mm}\\
{\it Claim 1: $\widetilde{\cal C}\cap \widetilde X_b$ is bounded towards the pseudo-convex end of $\widetilde X_b$ for every $b\in A$. }
\vspace{3mm}\\
Since  $ {\widetilde{\cal E}}$ is the (finite!) union of the closures $\widetilde {\cal C}$ of the  irreducible components $\widetilde{\cal C}_0$ of $\widetilde {\cal E}_0$, the claim implies obviously the second statement of the theorem.\\

We prove now Claim 1. Choose a point $b_0\in A$. Since $p$ is a locally trivial differentiable fiber bundle, it follows that there exists a smooth embedding $s_0:S^3\to X_{b_0}$ whose image $\Sigma_0$ does not disconnect $X$, and let $\widetilde s_0:S^3\to\widetilde X_{b_0}$ be a lift of $s_0$ to $\widetilde X_{b_0}$. We can deform $\widetilde s_0$ to get a smooth family of embeddings $(\widetilde s_b:S^3\to\widetilde{\cal X})_{b\in U}$, where $U$ is a connected open neighborhood of $b_0$ in $B$, such that $\widetilde s_{b_0}=\widetilde s_0$ and $s_b$ is an embedding of $S^3$ in $\widetilde X_b$ for every $b\in U$. We can regard this family as an embedding $\widetilde s:U\times S^3\to \widetilde {\cal X}_U:=\widetilde p^{-1}(U)$, whose image $\widetilde\Sigma_U$ separates $\widetilde {\cal X}_U$ in two connected components. Moreover,  $\widetilde {\cal X}_U$ can be written as an infinite union
$$\widetilde {\cal X}_U=\bigcup_{n\in\Z}  \widetilde {\cal X}_U^n \ ,
$$
where $\widetilde {\cal X}_U^n=  f^n(\widetilde {\cal X}_U^0)$ is a   manifold with two boundary components $  f^n(\widetilde\Sigma_U)$, $  f^{n+1}(\widetilde\Sigma_U)$.

Let $V\subsetint U$ an open neighborhood of $b_0$ with compact closure $\bar V\subset U$, and put 
$$\widetilde{\cal X}_{\bar V}:=\widetilde p^{-1}(\bar V)\ ,\ \widetilde\Sigma_{\bar V}:=\widetilde\Sigma_U\cap \widetilde{\cal X}_{\bar V}\ ,\ \widetilde {\cal X}_{\bar V}^n:=\widetilde {\cal X}_U^n\cap \widetilde{\cal X}_{\bar V} \ ,$$
and note that the spaces $\widetilde {\cal X}_{\bar V}^n=  f^n(\widetilde {\cal X}_{\bar V}^0)$ are all compact. Our claim follows from
\vspace{3mm}\\
{\it Claim 2: $\widetilde{\cal C}\cap \widetilde{\cal X}_{\bar V}$ is bounded towards the pseudo-convex end. }
\vspace{3mm}\\
Indeed if, by reductio ad absurdum, $\widetilde{\cal C}\cap \widetilde{\cal X}_{\bar V}$ were not bounded towards the pseudo-convex end, then the same would be true for  $  f^n(\widetilde{\cal C})\cap \widetilde{\cal X}_{\bar V}$  for every $n\geq 0$.  Therefore the irreducible   divisors $  f^n(\widetilde{\cal C})$ intersect  the compact subset $\widetilde{\cal X}^0_{\bar V}$ of $\widetilde {\cal X}$ for all sufficiently large $n\geq 0$. 

On the other hand using the inclusion   $\cup_{n\geq 0} f^n(\widetilde{\cal C})\subset\widetilde{\cal F}$ given by (\ref{inclusion}), we see that $\cup_{n\geq 0} f^n(\widetilde{\cal C})$  is a divisor. With this remark we can apply Lemma \ref{lmm} below, according to which it would follow that the sequence $(  f^n(\widetilde{\cal C}))_{n\geq 0}$ is finite, which is of course impossible, because the intersections of these divisors with $\widetilde p^{-1}(B_e)$  are pairwise different.\\

The third statement of the theorem follows now using the same arguments as in the proof of the similar statement in Theorem \ref{first}. 
\end{proof}
\begin{Lem} \label{lmm} Let $M$ be a connected complex manifold, and let $(D_n)_{n\in\mathbb{N}}$ be a sequence of effective irreducible divisors in $M$. Suppose that 
\begin{enumerate}
\item The union $D:=\bigcup_{n\in\mathbb{N}} D_n$ is a divisor of $M$, and
\item there exists a compact subspace $K\subset M$ such that $K\cap D_n\ne\emptyset$ for every $n\in \mathbb{N}$. 
\end{enumerate}
Then the set of divisors  $\{D_n|\  n\in\mathbb{N}\}$ is finite.
\end{Lem}
\begin{proof} If not, there would exist a subsequence $(D_{n_m})_{m\in\mathbb{N}}$ of  $(D_n)_{n\in\mathbb{N}}$ such that $D_{n_k}\ne D_{n_l}$ for every $k\ne l$.  Since all these divisors are irreducible, it follows that:
\vspace{2mm}\\  
{\it Remark: For $k\ne l$, the analytic set $D_{n_k}\cap D_{n_l}$ does not contain any codimension 1 irreducible component.} \\

For every $m\in\mathbb{N}$, let $x_m\in K\cap D_{n_m}$,  and let $x\in K$ be the limit of a convergent subsequence $(x_{m_s})_{s\in\mathbb{N}}$ of $(x_m)_{m\in\mathbb N}$. Let $U$ be a connected open neighborhood of $x$, so $U$ contains all the points $x_{m_s}$  with $s\geq s_0$ for a sufficiently large index $s_0\in\mathbb{N}$ .  For every $s\geq s_0$ we consider an irreducible component $E_s$ of the   intersection $D_{n_{m_s}}\cap U$, so $E_s$ is a {\it non-empty} effective divisor of $U$. Taking into account the remark above, it follows that $E_s\ne E_t$ for $s\ne t$.
 
 Therefore, the intersection $D\cap U=\left(\bigcup_{n\in\mathbb{N}} D_n\right)\cap U$ contains the infinite union $\bigcup_{s\geq s_0} E_s$ of pairwise distinct irreducible effective divisors $E_s$, so $D\cap U$ cannot be a divisor of $U$. This contradicts the first assumption of our hypothesis.

\end{proof}
 
Using the flatness of the obtained divisor $\widetilde{\cal E}$ over $B$ we see as in   Remark \ref{BM}  that, in the conditions of Theorem \ref{second}

\begin{Rem}\label{BMnew} For every point $b\in B$ one has
\begin{enumerate}
\item $c_1({\cal O}_{X_b}(\widetilde E_b))$ is the image of the Poincaré dual $PD(\widetilde e_b)\in H^2_c(\widetilde X_b,\Z)$ of $\widetilde e_b$ via the natural morphism   $H^2_c(\widetilde X_b,\Z)\to H^2 (\widetilde X_b,\Z)$.
\item $\widetilde E_b$ represents the image of $\widetilde e_b$ in the Borel-Moore homology group $H_2^{\rm BM}(\widetilde X_b,\Z)$.
\end{enumerate}
\end{Rem}
A natural problem  concerning the two extensions theorems Theorem \ref{first} and \ref{second} is to understand, for $b\in H_e$, the dependence of the obtained non-compact divisor $\widetilde{E}_b\subset \widetilde X_b$  on the lift $\widetilde{e}_b$ of $e_b$. It is easy to see that
\begin{Rem}  For a fixed  class $e_b$ with $e_b^2=\langle e_b,k_b\rangle=-1$, the  divisors $\widetilde{E}_b$, $\widetilde{E}'_b$ associated with two different lifts $\widetilde{e}_b$, $\widetilde{e}_b'$  of $e_b$ are mapped on each other by a power $f^n$ of the generator $f$ of the automorphism group $\mathrm{Aut}_{X_b}(\widetilde{X}_b)$. 
\end{Rem}

When the fiber $X_b$ is a GSS surface  much more can be said: in the next section we shall see that, if $X:=X_b$ is a GSS surface, then the effective divisors $\widetilde{E}:=\widetilde{E}_b\subset\widetilde{X}$ are  always given by   series of {\it compact} curves. Moreover, in this case we shall see that there exists a natural {\it total} order in the set $\{\widetilde{e}\in H_2(\widetilde{X},\Z)|\ \alpha_*(\widetilde{e})^2=\langle\alpha_*(\widetilde{e}),k_b\rangle=-1\}$, and endowed with this order, this set (which is a basis of $H_2(\widetilde{X},\Z)$) can be identified with $(\Z,\leq)$. The map $\widetilde{e}\mapsto \widetilde E$ which associates to such a class the corresponding effective divisor, is monotone decreasing, in the sense that $\widetilde E\leq \widetilde E'$ whenever $\widetilde e'\leq \widetilde e$ (see Proposition \ref{series}, Remark \ref{unicity}). At this moment we cannot prove these properties for a fiber $X_b$ ($b\in H_e$) which is not assumed to be a GSS surface.

\section{The case of GSS surfaces. The fundamental series} \label{GSSsurfaces}

We begin with the following
\begin{Def} Let $X$ be a compact complex surface. We say that $X$ contains a global spherical shell (GSS), or that $X$ is a GSS surface, if there exists a biholomorphic map $\f:U\to X$ from a neighborhood $U\subset \bb C^2\setminus\{0\}$ of the sphere $S^3$ into $X$ such that $X\setminus \f(S^3)$ is connected.
\end{Def}
Primary Hopf surfaces   are the simplest examples of GSS surfaces. The differential topological type of a GSS surface $X$ with $b_2(X)=n$ is  $(S^1\times S^3)\#n\overline{\mathbb{P}}^2$, so $H_2(X,\Z)$ admits an  (unordered)    basis $\mathfrak{B}\subset H_2(X,\Z)$ trivializing the intersection form $q_X$ of $X$, i.e. such that 
$$e' \cdot e''=\left\{
\begin{array}{ccc}
0&\rm if & e'\ne e''\\
-1&\rm if &e'=e''
\end{array}  \right.\ .$$

  Decomposing the Chern class $c_1({\cal K}_X)=-c_1(X)\in H^2(X,\Z)$ with respect to the Poincaré dual basis $\mathfrak{B}^\vee=\{PD(e)|\ e\in\mathfrak{B}\}$ of $H^2(X,\Z)$ we get 
$$c_1({\cal K}_X)=\sum_{e\in\mathfrak{B}} k_e  PD(e)$$
 with $k_e\equiv 1$ mod 2 (because $c_1({\cal K}_X)$ is a characteristic element) and $\sum_{e\in\mathfrak{B}} k_e^2=-c_1(X)^2= c_2(X)= n$. This shows that $k_e\in\{\pm 1\}$  so, replacing some of the elements $e\in\mathfrak{B}$ by $-e$ if necessary, we may assume that 
\begin{equation}
c_1({\cal K}_X)=\sum_{e\in\mathfrak{B}}    PD(e) \ ,
\end{equation}
or equivalently
\begin{equation}
\forall e\in\mathfrak{B}\ ,\  \langle e ,c_1({\cal K}_X)\rangle=-1 \ .   
\end{equation}

There exists a unique unordered $q_X$-trivializing basis $\mathfrak{B}$ with this property, and this basis will be called {\it the standard basis} of $H_2(X,\Z)$.   Note that the homology class of any simply exceptional divisor of $X$ (if such a divisor exists) is an element of its standard basis.\\

For surfaces $X$ with GSS, the second Betti number  $b_2(X)$ is equal to the number of rational curves in $X$.     A  {\it marked GSS surface} is a pair $(X,C_0)$ consisting of GSS surface with $b_2(X)\ge 1$ and a rational curve $C_0$ in $X$.\\

Any minimal marked GSS surface   $(X,C_0)$ with $n=b_2(X)\ge 1$ can be obtained using a simple 2-step construction consisting of an iterated blow  up  followed by a holomorphic surgery. More precisely, let $\Pi=\Pi_0\circ\cdots\circ\Pi_{n-1}:B^\Pi\to B$ be a sequence of $n$ blowing ups such that 
\begin{itemize}
\item[-] the first blowing up $\Pi_0:B^{\Pi_0}\to B $ blows up the origin $O_{-1}:=(0,0)$ in the 2-dimensional unit ball $B\subset\C^2$,
\item[-] for $i=0,\ldots,n-2$ the  blowing up $\Pi_{i+1}: B^{\Pi_0\circ\dots\circ\Pi_{i+1}}\to B^{\Pi_0\circ\dots\circ\Pi_{i}}$ blows up a point $O_{i}\in  \Pi_{i}^{-1}(O_{i-1})$ in the surface $B^{\Pi_0\circ\dots\circ\Pi_{i}}$ obtained at the previous step.

\end{itemize}

Applying the same sequence of blowing ups to   a ball $B(r)$ of radius $r$ and to its closure $\bar B(r)$ we obtain complex surfaces (respectively compact complex surfaces with boundary)  which will be denoted  by  $B(r)^\Pi$, respectively  $\bar B(r)^\Pi$.

Let now $\s:\bar B\to B^\Pi$ an embedding which extends to a biholomorphism from a neighborhood of $\bar B$ onto a small open ball in $B^\Pi$such that $\s(0)=O_{n-1}\in \Pi_{n-1}^{-1}(O_{n-2})$. One can associate to the pair $(\Pi,\s)$ a minimal surface $X=X(\Pi,\sigma)$ by removing from $\bar B^\Pi$ the image $\sigma(B)$, and identifying the two boundary components $\partial \bar B^\Pi$, $\sigma(\partial \bar B)$ of the compact manifold with boundary $A:=\bar B^\Pi\setminus \sigma(B)$  by $\sigma\circ \Pi$, the holomorphic structure on the resulting differentiable manifold being defined such that the natural surjective locally diffeomorphic map 
$$q_{\Pi,\s}: B(1+\epsilon)^\Pi\setminus \sigma(\bar B(1-\epsilon))\longrightarrow X(\Pi,\s)\  .
$$  
is a local biholomorphism. One can prove that for any marked surface $(X,C_0)$ there exists a pair $(\Pi,\s)$ and a biholomorphism $X\textmap{\simeq} X(\Pi,\sigma)$ such that $C_0$ corresponds to the image via  $q_{\Pi,\s}$ of the   first exceptional curve $\Pi_0^{-1}(O_{-1})$. We will identify $X$ with $X(\Pi,\sigma)$ via this biholomorphism.

For any $i\in\{0,\dots,n-2\}$ consider a small closed ball $\bar B_i$ centered at   $O_i\in B^{\Pi_0\circ\dots\circ\Pi_{i}}$ and denote by $S_i$ the lift of its boundary to $B^\Pi$.  It is also convenient to denote by $S_{-1}:=\partial \bar B^\Pi$, $\bar B_{n-1}:=\sigma(\bar B)$, and $S_{n-1}:=\partial \bar B_{n-1}=\sigma( \partial \bar B)$. In this way $A$ decomposes as the union
$$A=\bigcup_{i=0}^{n-1}\mathfrak{A}_i\ ,
$$
where each $\mathfrak{A}_i$ is a compact surface with two boundary components: a  strictly  pseudo-convex component $\partial_-(\mathfrak{A}_i)\simeq S_{i-1}$ and the strictly pseudo-concave component $\partial_+(\mathfrak{A}_i)\simeq S_i$. Every $\mathfrak{A}_i$ is biholomorphic to the manifold obtained from a blown up closed ball by removing a small open ball centered at a point of the exceptional divisor.

The universal covering space $\widetilde X$ can be obtained as an infinite union of copies  $A_k$ of $A$, the pseudo-convex boundary  component   of   $A_k$ being identified with the pseudo-concave boundary component of the previous annulus $A_{k-1}$ via $\sigma\circ\Pi$. We will identify $A_0$ with $A$, so $A_0$  decomposes as  $\bigcup_{i=0}^{n-1}\mathfrak{A}_i$; correspondingly we get decompositions  
$$\forall j\in\Z\ ,\ A_j=\bigcup_{i=0}^{n-1}\mathfrak{A}_{nj+i}\ ,\ \widetilde X=\bigcup_{s\in\Z}\mathfrak{A}_s\ ,$$
 where $\mathfrak{A}_{nj+i}$ is a copy of $\mathfrak{A}_i$ for every $i\in\{0,\dots,n-1\}$, and $\widetilde X$ is obtained by identifying the pseudo-convex boundary component   $\partial_-(\mathfrak{A}_{s})$ of any  $\mathfrak{A}_{s}$ with the pseudo-concave boundary component $\partial_+(\mathfrak{A}_{s-1})$ of the previous piece, in the obvious way. For every $k\in\Z$ we put
$$\widetilde X_k:=\bigcup_{s\leq k} \mathfrak{A}_s\ .
$$
Note that $\widetilde X_{nj+n-1}= \bigcup_{j'\leq j}A_{j'}$. The boundary $\partial \widetilde X_k$   is the strictly pseudo-concave 3-sphere $\partial_+(\mathfrak{A}_k)\simeq  \partial_+(\mathfrak{A}_{k-n\left[\frac{k}{n}\right]})\simeq S_{k-n\left[\frac{k}{n}\right]}$. We denote by $\widehat X_k$ the surface  obtained by gluing (a copy  of)  $\bar B_{k-n\left[\frac{k}{n}\right]}$ along this boundary. For every $k< l$ we have obtain a commutative diagram
$$
\begin{array}{ccc}
\widetilde X_k&\stackrel{\iota_{k,l}}\longrightarrow &\widetilde X_l\\
\iota_k\downarrow\ \ \ \ & &\ \ \ \downarrow\iota_l\\
\widehat X_k&\stackrel{\pi_{k,l}}\longleftarrow &\widehat X_l
\end{array}\ ,
$$
where $\iota_k$, $\iota_l$, $\iota_{k,l}$ are the obvious inclusions, and $\pi_{k,l}$ is an order  $l-k$ iterated blow  up at the center of the ball $\bar B
_{k-n\left[\frac{k}{n}\right]}$.  For $k<l<m$ one has obviously $\iota_{l,m}\circ\iota_{k,l}=\iota_{k,m}$, $\pi_{k,l}\circ\pi_{l,m}=\pi_{k,m}$. Note that, for $i\in\{1,\dots,n-1\}$ the blow up $\pi_{nj+i-1,nj+i}$   corresponds to $\Pi_i$ via obvious identifications, whereas $\pi_{nj-1,nj}$ corresponds to $\Pi_0\circ\sigma^{-1}$. We will denote by $E_k$ the exceptional curve of the blow up $\pi_{k-1,k}:\widehat X_k\to \widehat X_{k-1}$. The pre-image $\iota_k^{-1}(E_k)$ is of course non-compact, but its Zariski closure in $\widetilde X$ is a compact rational curve $\widetilde C_k$ which can be explicitly obtained in the following way  (see \cite{D1} for details): 
\begin{Rem} \label{s} There exists $s\in \{1,\dots,n+1\}$ such that the proper transform $E_k^s$ of $E_k$ in $\widehat X_{k+s}$ is contained in the image of $\iota_{k+s}$, so $\widetilde C_k:=\iota_{k+s}^{-1}(E_k^s)$ is a compact rational curve in $\widetilde X$, which can be identified with the Zariski closure of $\iota_k^{-1}(E_k)$.
\end{Rem}

Note now that the group morphisms $H_2(\iota_k):H_2(\widetilde X_k,\Z)\to H_2(\widehat X_k,\Z)$ are isomorphisms, and that  via the  isomorphisms $H_2(\iota_k)$, $H_2(\iota_{k-1})$, the monomorphism $H_2(\iota_{k-1,k})$ defines a right splitting of the short exact sequence
$$0\to \Z[E_k]\to H_2(\widehat X_k,\Z)\textmap{H_2(\pi_{k-1,k})} H_2(\widehat X_{k-1},\Z)\to 0 \ .
$$
We denote by $\widetilde e_k$ the image  in $H_2(\widetilde X_k,\Z)$ of the class $[E_k]$ via $H_2(\iota_k)^{-1}$; to save on notations we will use the same symbol for the images of $\widetilde e_k$ in $H_2(\widetilde X_l,\Z)$ ($l\geq k$) and in $H_2(\widetilde X,\Z)$.  With these conventions we can write
$$H_2(\widetilde X_k,\Z)=\bigoplus_{s\leq k} \Z \widetilde e_s\ ,\ H_2(\widetilde X,\Z)=\bigoplus_{s\in\Z} \Z \widetilde e_s
$$
Using Remark \ref{s} we see immediately that, via this decomposition, the homology class $[\widetilde C_k]\in H_2(\widetilde X,\Z)$ decomposes as
\begin{equation}\label{dec} [\widetilde C_k]=\widetilde e_k-\sum_{i=1}^{s_k} \widetilde e_{k+i} \hbox { where } 1\leq s_k\leq n+1\ .
\end{equation} 
Note that one has obviously
\begin{equation}\label{sk} s_k=-\widetilde C_k^2-1\ .
\end{equation}

The $n$ compact curves of the minimal GSS surface $X=X(\Pi,\s)$ are just the projections $C_k:=\alpha(\widetilde C_{nj+k})$ of the compact curves we obtained in the universal cover $\widetilde X$. Note also that the set  $\{e_k:=\alpha_*(\widetilde e_k)|\ k\in\{0,\dots,n-1\}\}$ is precisely the standard basis $\mathfrak{B}$ of $X$.

We will show now that any class $\widetilde e_i$ decomposes formally in a well defined way as an  infinite  series of classes of compact curves which is bounded towards the pseudo-convex end. These formal identities  correspond to equalities in the Borel-More homology of $\widetilde X$. We will begin with several interesting examples:

\begin{Ex} \label{Enoki}  Enoki surfaces. 
\end{Ex}

An Enoki surface $X$ has    a cycle $\sum_{i=0}^{n-1}C_i$ of $n=b_2(X)$ rational curves. We have
$$[C_i]= e_i-e_{i+1}, \quad i\in \Z_n\ .$$

In the universal cover $\widetilde X$ we have curves $\widetilde C_i$ representing the 2-homology classes $[\widetilde C_i]= \widetilde e_i-\widetilde e_{i+1}, \  i\in\bb Z$, therefore we get the decomposition
$$\widetilde e_i= \sum_{j=0}^\infty [\widetilde C_{i+j}] \ .$$

\begin{Ex} 
 Minimal GSS surfaces $X$ with $b_2(X)=1$. 
\end{Ex} 

There are two classes of minimal GSS surfaces  $X$ with $b_2(X)=1$:

\begin{itemize}
\item[-] Enoki surfaces with $b_2=1$,
\item[-] Inoue-Hirzebruch surfaces with one cycle, 
\end{itemize}

The first class has been treated above, so we discuss the second one. 
The curves in the universal cover $\widetilde X$ of an Inoue-Hirzebruch surface $X$ with $b_2(X)=1$ form two chains of rational curves 
$$\sum_{i\in\bb Z} \widetilde C_{2i},\quad \sum_{i\in\bb Z} \widetilde C_{2i+1}$$
 with $[\widetilde C_i]= \widetilde e_i- (\widetilde e_{i+1}+\widetilde e_{i+2})$. Using the Fibonacci sequence $(u_n)_{n\in\mathbb N}$   given by
$$u_0=1,u_1=1,\ u_n=u_{n-2}+u_{n-1}, n\ge 2$$
we get
$$\widetilde e_i=\sum_{n\ge 0}u_n\widetilde C_{i+n}\ .$$

\begin{Ex} 
  Minimal GSS surfaces $X$ with  $b_2(X)=2$. 
\end{Ex} 

There are four classes of minimal GSS surfaces $X$ with  $b_2(X)=2$.
\begin{itemize}
\item[-] Enoki surfaces with $b_2=2$, 
\item[-]  Intermediate surfaces, 
\item[-]  Inoue-Hirzebruch surfaces with one cycle,
\item[-] Inoue-Hirzebruch surface with two cycles consisting of a rational curve with a double point.
\end{itemize}
The case of Enoki surfaces has been treated above  for any $b_2>0$.\\
\paragraph {\it  Intermediate surfaces.} An intermediate surface with $b_2=2$ has a rational curve with a double point  $C_0$ and a non-singular rational curve $C_1$ with
$$C_0^2=-1,\quad C_1^2=-2,\quad C_0C_1=1\ . $$
Using our conventions and notations we get decompositions
$$[C_0]= -e_1,\  [C_1]= e_1-e_0\ .$$
 In the universal covering space $\widetilde X$ we have an infinite chain of curves
$\sum_{i\in\bb Z} \widetilde C_{2i}$ with pairwise disjoint trees $\widetilde C_{2i+1}$, $i\in \bb Z$ which decompose as follows:
$$[\widetilde C_{2i}]= \widetilde e_{2i}- \widetilde e_{2i+1}-\widetilde{e}_{2i+2}\ , \  [\widetilde C_{2i+1}]= \widetilde e_{2i+1}- \widetilde e_{2i+2}\ .$$
We have then 
$$\begin{array}{lcl}\widetilde e_{2i}&=&\displaystyle \sum_{j=0}^\infty 2^j\Bigl\{(\widetilde e_{2(i+j)}-\widetilde e_{2(i+j)+1}-\widetilde e_{2(i+j)+2})+(\widetilde e_{2(i+j)+1}-\widetilde e_{2(i+j)+2})\Bigr\}\\
&&\\
&=&\displaystyle \sum_{j=0}^\infty 2^j\Bigl\{\widetilde C_{2(i+j)}+\widetilde C_{2(i+j)+1}\Bigr\}
\end{array}$$
$$ \widetilde e_{2i+1}=(\widetilde e_{2i+1}-\widetilde e_{2i+2})+\widetilde e_{2i+2}=[\widetilde C_{2i+1}]+\sum_{j=0}^\infty 2^j\Bigl\{[\widetilde C_{2(i+j+1)}]+[\widetilde C_{2(i+j+1)+1}]\Bigr\}\ .
 $$
\paragraph{\it Inoue-Hirzebruch surfaces with $b_2=2$ and one cycle}

Such a surface has a cycle $C_0+C_1$ of non-singular rational curves, where
$$C_0^2=-4 \ ,\ C_1^2=-2\ ,\ C_0C_1=2\ .$$
More precisely
$$[C_0]= -2e_1\ ,\  [C_1]= e_1-e_0\ .$$
In  the universal cover $\widetilde X$   the homology classes of the compact curves  $\widetilde C_i$ decompose as 
$$[\widetilde C_{2i}]= \widetilde e_{2i}-\widetilde e_{ 2i+1} -\widetilde e_{2i+2}-\widetilde e_{2i+3}\ ,\ [\widetilde C_{2i+1}]= \widetilde e_{2i+1}-\widetilde e_{2i+2}\ . $$
These curves form two disjoint infinite chains of rational curves 
$$\sum_{i\in\Z} \widetilde C_{4i} + \widetilde C_{4i+3}\ ,\ \sum_{i\in\Z}\widetilde C_{4i+1} + \widetilde C_{4i+2}\ .$$

We define by induction the following sequences of positive integers $(a_j)_{j\in\mathbb N}$, $(b_j)_{j\in\mathbb N}$ where $a_0=b_0=1$, and
$$\left(\begin{array}{c}a_j\\b_j\end{array}\right)=A^j\left(\begin{array}{c}a_0\\b_0\end{array}\right)\hbox{ with }  A:=\left(\begin{array}{cc}1&1\\2&1\end{array}\right)\ .$$
Then
$$\begin{array}{lcl}
\widetilde e_{2i}&=&\displaystyle \sum_{j=0}^\infty \Bigl\{a_j\bigl(\widetilde e_{2(i+j)}-\widetilde e_{2(i+j)+1}-\widetilde e_{2(i+j)+2}-\widetilde e_{2(i+j)+3}\bigr)+ b_j\bigl(\widetilde e_{2(i+j)+1}-\widetilde e_{2(i+j)+2}\bigr)\Bigr\}\\
&&\\
&=&\displaystyle\sum_{j=0}^\infty a_j [\widetilde C_{2(i+j)}] + b_j [\widetilde C_{2(i+j)+1}]\\
&&\\
& =& [\widetilde C_{2i}]+[\widetilde C_{2i+1}] + 2[\widetilde C_{2i+2}]+3[\widetilde C_{2i+3}]+5[\widetilde C_{2i+4}]+7[\widetilde C_{2i+5}]+12[\widetilde C_{2i+6}]+\cdots
\end{array}$$
$$\widetilde{e}_{2i+1}=(\widetilde{e}_{2i+1}-\widetilde{e}_{2i+2})+\widetilde{e}_{2i+2}=[\widetilde{C}_{2i+1}]+\sum_{j=0}^\infty a_j[\widetilde{C}_{2(i+j+1)}]+b_j[\widetilde C_{2(i+j+1)+1}]\ .
$$

\paragraph{\it Inoue-Hirzebruch surfaces with $b_2=2$ and two cycles}
In this case we have two cycles  $C_0$, $C_1$ consisting of a rational curve  with double point. The intersection numbers are 
$$C_0^2=-1,\quad C_1^2=-1\ ,$$
and the decompositions of $[C_i]$ with respect to the standard basis are
$$[C_0]= -e_1,\quad [C_1]= -e_0\ .$$
In 
$\widetilde X$  we have two disjoint chains of rational curves 
$\sum_{i\in\bb Z}\widetilde C_{2i}$, $\sum_{i\in\bb Z}\widetilde C_{2i+1}$,
where
$$[\widetilde C_{2i}]= \widetilde e_{2i}-\widetilde e_{2i+1}- \widetilde e_{2i+2}\ ,  [\widetilde C_{2i+1}]= \widetilde e_{2i+1}-\widetilde e_{2i+2}-\widetilde e_{2i+3}\ .$$
Therefore  denoting again by $(u_j)_{j\in\mathbb N}$ the Fibonacci sequence used before we obtain:
$$\begin{array}{lclcl}
\widetilde e_{2i}&=&\displaystyle\sum_{j\ge 0} u_j(\widetilde e_{2i+j}-\widetilde e_{2i+j+1}- \widetilde e_{2i+j+2})&=&\displaystyle\sum_{j\ge 0}u_j[\widetilde C_{2i+j}]\\
&&&&\\
\widetilde e_{2i+1}&=&\displaystyle \sum_{j\ge 0} u_j(\widetilde e_{2i+j+1}-\widetilde e_{2i+j+2} - \widetilde e_{2i+j+3})&=&\displaystyle\sum_{j\ge 0}u_j[\widetilde C_{2i+j+1}]\ .
\end{array}$$

The next result deals with the general case:
 
\begin{Prop} \label{series} Let $X=X(\Pi,\s)$ be a GSS surface with $b_2(X)=n\ge 1$ and $(\widetilde e_i)_{i\in\Z}$ be the standard  base of $H_2(\widetilde X,\Z)$ constructed above. Then for any $i\in\Z$ there exists a  well defined infinite series with  positive integer coefficients $\sum_{j\ge 0} \a^i_{j}[\widetilde C_{i+j}]$ whose sum is the image of $\widetilde e_i$ in the Borel-Moore homology group $H_2^{\rm BM}(\widetilde X,\Z)$.
Moreover, the  sequence $(\a^i_{j})_{j\ge 0}$ is always   increasing, has $\alpha^i_0=1$, and 
 $$\left\{ \begin{array}{lll}
 \forall j,\quad \a^i_{j}=1&{ \it iff}& X\ {is\ a\ Enoki\ surface\ ,}\\
 &&\\
 \lim_{j\to \infty}\a ^i_j=\infty&{\it  iff}& X\ { is\ not\ a\ Enoki\ surface\ .}
 \end{array}\right.$$ 
\end{Prop}
\begin{proof} We determine the coefficient $\alpha^i_k$ of $[\widetilde C_{i+k}]$ by induction on $k\ge 0$ such that we have the  congruence
\begin{equation}\label{cong}
\widetilde e_i\equiv \sum_{j=0}^k \alpha^i_j[\widetilde C_{i+j}]
\hbox{ modulo  classes } \widetilde e_{j}\ ,\   j>k\ . \end{equation} 
For $k=0$ we use formula (\ref{dec}) to obtain
$$  [C_i]=\widetilde e_i-\sum_{j=1}^{s_i}\widetilde e_{i+j}\equiv  \widetilde e_i \hbox{ mod classes }\widetilde e_{i+j}\ ,\ j>0\ ,$$
so we get a well-defined solution $\alpha^i_0=1$.

Suppose now that the development is determined till   $k\in\mathbb N$ so that
$$\widetilde e_i\equiv \sum_{j=0}^k \a^i_{j}[\widetilde C_{i+j}]=\sum_{j=0}^k \a^i_{j}\big(\widetilde e_{i+j}-\sum_{l=1}^{s_{i+j}}\widetilde e_{i+j+l}\big)  \hbox{ mod classes }\widetilde e_{i+j}\ ,\ j>k\ .$$
We define $J_{k+1}^i=\big\{j\in\{0,\dots,k\}\mid\     i+j+1\leq i+k+1\leq i+j+ s_{i+j}\big \}$.  With this notation we obtain a well defined solution 
\begin{equation}\label{identity}
\a^i_{k+1}=\sum_{l\in J_{k+1}^i}\a^i_{l}  \end{equation}
which guarantees the congruence (\ref{cong}) for $k+1$. \\

We prove now that the obtained sequence $(\a^i_j)_{j\in\mathbb N}$ is increasing as claimed. Since for any $l\in\bb Z$ it holds  $l+1\leq l+s_l$, we have  $k+1\in J_{k}^i$. Therefore  $\a^i_{k}$ intervenes in the decomposition (\ref{identity}), so  $\a^i_{k+1}\ge \a^i_{k}$. 

If $X$ is an Enoki surface, then $\a_{i+j}=1$ for every $j\ge 0$ as we have seen in Example \ref{Enoki}. If $X$ is not an Enoki surface there is at least one curve $\widetilde C_{i+j}$ such that $\widetilde C_{i+j}^2\le -3$, hence, using again formula (\ref{dec}) we see that    $\widetilde e_{i+j+2}$ appears in the decompositions of the  curves $\widetilde C_{i+j}$ and $\widetilde C_{i+j+1}$, hence    $j,\ j+1\in J^i_{j+2}$ and $\a_{j+2}\geq \a_j+\a_{j+1}>\a_{j+1}$. Since the configuration of the curves in $\widetilde X$ is periodic, we have $\widetilde C^2_{i+j+kn}=\widetilde C^2_{i+j}\leq -3$ for every $k\in\mathbb N$, so $\a_{j+kn+2}>\a_{j+kn+1}$, which proves the result. 
\end{proof}  
\begin{Rem}\label{unicity} The decomposition $\widetilde e_i=\sum_{j\ge 0} \a^i_{j}[\widetilde C_{i+j}]$  given by Proposition \ref{series} is the only decomposition of the image of   $\widetilde e_i$ in $H_2^{\rm BM}(\widetilde X,\Z)$  as the sum of a series -- bounded  towards the pseudo-convex end  -- of classes of compact curves.
\end{Rem}
\begin{proof} Indeed, it suffices to see that if a sum $\sum_{i\geq k} a_i[\widetilde C_i]$ vanishes in  $H_2^{\rm BM}(X,\Z)$, then all coefficients $a_i$ vanish. This follows easily using the Poincaré duality isomorphism $PD: H_2^{\rm BM}(\widetilde X,\Z)\textmap{\simeq} H^2(\widetilde X,\Z)$. Using the geometric interpretation of the Poincaré duality in terms of intersection numbers we get
$$0=\langle PD(\sum_{i\geq k} a_i[\widetilde C_i]), \widetilde e_k\rangle=-a_k \ .
$$
By induction we get $a_i=0$ for all $i\geq k$.
\end{proof}
Note that for an Enoki surface $X$  the classes $[\widetilde C_i]$ are linearly independent in $H_2(\widetilde X,\Z)$, but $\sum_{i\in\Z} [\widetilde C_i]=0$ in $H_2^{\rm BM}(\widetilde X,\Z)$. 
\vspace{3mm}

Let  now $p:{\cal X}\to B$ be a  holomorphic family of GSS surfaces and let  %
$$e=(e_b)_{b\in B}\in H^0(B,\underline{H}_2)\ ,\ \widetilde e=(\widetilde e_b)_{b\in B}\in H^0(B,\underline{\widetilde H}_2)$$
 be families of homology classes as in the first extension theorem Theorem \ref{first}. 

Our goal now is to identify explicitly the effective divisor $\widetilde E_{b_0}$ for $b_0\in H_e$. As we explain in the introduction this divisor is interesting because it can be written as a limit of simply exceptional divisors $\widetilde E_{b_n}$ for a sequence $b_n\to b_0$, $b_n\in B_e$.  

We suppose for simplicity that $X_{b_0}$ is minimal  because, if not, we can blow down the exceptional divisors in $X_{b_0}$ (and their deformations in $X_{b}$ for $b$ around $b_0$) and reduce the problem to the case of a minimal central fiber.

\begin{Th} \label{GSScase} Under the assumptions and with the notations above, the effective  divisor $\widetilde E_{b_0}$ is the sum of the infinite series   $\sum_{j\ge 0} \a^i_{j}\widetilde C_{i+j}$ given by Proposition \ref{series}.
\end{Th}
\begin{proof}  Our extension theorem yields a divisor $\widetilde {\cal E}\subset \widetilde{\cal X}$ flat over $B$ whose fiber over $b_0$ is $\widetilde E_{b_0}$, so the class defined by $\widetilde E_{b}$ in Borel-Moore homology  is $\widetilde e_b$  for any $b\in B$. The result follows directly from  Remark \ref{unicity} and Lemma \ref{final} below.
\end{proof}
\begin{Lem}\label{final} Let $X$ be a GSS surface with $b_2(X)>0$ and $\alpha:\widetilde X\to X$ its universal cover. Then $\widetilde X$ does not contain any irreducible   non-compact 1-dimensional analytic subset bounded towards the pseudo-convex end. 
\end{Lem}
\begin{proof}

Suppose that there exists such a subspace $S\subset\widetilde X$. We can suppose  that $X$ is minimal and was identified with $X(\Pi,\s)$ as explained above. Since $S$ is bounded towards the  pseudo-convex end we can find $k$ sufficiently large  such that $S$ does not intersect the closure of $\widetilde X_k$.

We have a natural surjective {\it holomorphic} map $q_k:\widetilde X\to \widehat X_k$ which contracts all compact curves $\widetilde C_l$ for $l>k$ to a point $\widehat O_k\in \widehat X_k$ (the center of ball $B_{k-n\left[\frac{k}{n}\right]}$ involved in the construction of $\widehat X_k$) and defines a biholomorphism
$$\widetilde X\setminus  \bigcup_{l>k} \widetilde C_l\textmap{\simeq} \hat X_k\setminus
\{\widehat O_k\}
$$
(see \cite{D1}).  Since $S$ does not intersect the closure of $\widetilde X_k$ the image $ S_0:=q_k(S)$ is contained in the complement of this closure in $\hat X_k$, which is the open ball $B_{k-n\left[\frac{k}{n}\right]}$.

But then $S_0\setminus\{\widehat O_k\}$ is  a closed 1-dimensional subspace of a punctured open 2-ball, so its closure in this ball  will be  a closed 1-dimensional subspace of the ball, by the second Remmert-Stein theorem. This extension would be compact, which is of course impossible.

\end{proof}


\def\wt{\widetilde}

Coming back to our second extension theorem Theorem \ref{second}, and taking into account Theorem \ref{GSScase} we  see that, when a fiber $X_b$ ($b\in A)$ of the family is a GSS surface, the obtained effective divisors $\wt{E}_b$ representing a lift $\wt{e}_b$ are always given by   series of {\it compact} curves. We end our article with a theorem (mentioned in the abstract and the introduction) which gives an interesting motivation of our results: conversely, if in the conditions of Theorem \ref{second}, for a point $b\in A$ all the irreducible components of the  divisors $\wt{E}_b$ are compact, then $X_b$ has $b_2$ rational curves, so it is a GSS surface. We believe that {\it the irreducible components of the  divisors $\wt{E}_b$ are always compact}, which would have a very important consequence:  
the main conjecture {\bf C} holds for any class VII surface which fits as the central fiber of a bidimensional family ${\cal X}\to\Delta\subset\C^2$ whose fibers 
$X_z$, $z\ne 0$ are GSS surfaces. Therefore, for proving conjecture {\bf C} for this class of surfaces, it suffices to prove that for  a sequence $(b_n)_n$ of $\Delta\setminus\{0\}$ converging to $0$, open irreducible components  cannot appear in the limit process $\wt{E}_{b_n}\to\wt{E}_0$, but only {\it infinite bubbling} i.e. the appearance of infinitely many compact irreducible components.
\\

Let $X$ be a class VII surface with $b_2>0$ having the oriented topological type of $(S^1\times S^3)\# n\overline{\mathbb P}^2$. In the same way as explained at the beginning section for GSS surfaces we obtain a unique  (unordered)    basis $\mathfrak{B}\subset H_2(X,\Z)$ trivializing the intersection form $q_X$ of $X$ such that $c_1({\cal K}_X)=\sum_{e\in\mathfrak{B}}    PD(e)$. This basis   will be  called again {\it the standard basis} of $H_2(X,\Z)$.   Note that the homology class of any simply exceptional divisor of $X$ (if such a divisor exists) is an element of this standard basis.
Any compact curve $C\subset \widetilde X$ is a smooth rational curve. Indeed, its projection $\alpha(C)$ on $X$ is either a  
\begin{enumerate}
\item smooth rational curve, or
\item  a rational curve with a simple singularity, or
\item or an elliptic curve 
\end{enumerate}
(see \cite{N1} p. 399). But in the last two cases $\alpha(C)$ would be a cycle in $X$, and  it is known  that the image of its fundamental group  in $\pi_1(X)$  of a cycle is infinite (see \cite{N1} p. 404).  In the second case $\alpha^{-1}(\alpha(C))$ is a countable union of smooth  rational curves, and the third case  is excluded, because the pre-image of such a cycle does not contain any compact curve.  

Therefore
\begin{equation}\label{genus}0=g_a(C)=1+\frac{1}{2}\langle c_1({\cal K}_{\tilde X}(C)),C\rangle
\end{equation}
  Since all the elements of the standard basis $\mathfrak{B}$ are represented by 2-spheres, they can be lifted to $\widetilde X$, so we get a basis $\widetilde{\mathfrak{B}}$  of $H_2(\widetilde X,\Z)$ which is mapped on  $\mathfrak{B}$.   We can identify naturally $H^2(\widetilde{X},\Z)$ with $\prod_{\widetilde{e}\in\widetilde{\mathfrak B}}\Z PD(   \widetilde{e})$, and we obtain 
$$c_1({\cal K}_{\tilde X})=\alpha^*(c_1({\cal K}_X))=\alpha^*(\sum_{e\in{\mathfrak B}} PD(e))=\sum_{  \widetilde{e}\in \widetilde{\mathfrak{B}}} PD(   \widetilde{e})\ .
$$
Decomposing $[C]=\sum_{ \widetilde{e}\in\widetilde{\mathfrak B}} x_{\widetilde{e}} \widetilde{e}$ (with finitely many non-zero coefficients $x_{\widetilde{e}}$) and  using the identities $\langle PD(\widetilde{e}),\widetilde{f}\rangle=-\delta_{\widetilde{e}\widetilde{f}}$ equation (\ref{genus}) becomes:
$$\sum_{\widetilde{e}\in \widetilde{\mathfrak B}} x_{\widetilde{e}}(1+x_{\widetilde{e}})=2\ ,
$$
which can hold only if there exists $\widetilde{e}_C\in \widetilde{\mathfrak B}$ and a finite set $\widetilde{\mathfrak B}_C\subset\widetilde{\mathfrak B}$ with $\widetilde{e}_C\not\in \widetilde{\mathfrak B}_C$ such that either:
\begin{enumerate}
\item[a)] $[C]=  \widetilde{e}_C-\sum_{\widetilde{e}\in \widetilde{\mathfrak B}_C} \widetilde{e}$ , or
\item[b)]  $[C]=-2 \widetilde{e}_C-\sum_{\widetilde{e}\in \widetilde{\mathfrak B}_C} \widetilde{e}$\ .
\end{enumerate}

We  say that a curve $C\subset \wt{X}$ is of type (a) or (b) if the decomposition of $[C]$ has the form a) or b) respectively.

\begin{Th}\label{converse} Let $X$ be a minimal class VII surface with $b_2(X)>0$ having the topological type $(S^1\times S^3)\# n\overline{\mathbb P}^2$. Suppose that for every $e\in\mathfrak{B}$ there exists a lift $\wt{e}\in\wt{\mathfrak{B}}$ and an effective divisor $\wt E\subset\wt{X}$ representing the image of $\wt{e}$ in  $H_2^\mathrm{BM}(X,\Z)$ and whose irreducible components are all compact. Then $X$ has $b_2(X)$ rational curves, so it is a GSS surface.
\end{Th}
\begin{proof} The Borel-Moore homology group $H_2^\mathrm{BM}(\wt X,\Z)$ can be identified with $\prod_{\wt{e}\in\wt{\mathfrak{B}}}\Z\wt{e}$. Since $\wt E$ represents the class $\wt{e}$, it follows that there exists an irreducible  component $C_e$ of $\wt E$ which has type (a) with $\wt{e}_{C_e}=e$. The image $D_e:=\alpha(C_e)$ is either a smooth rational curve, or a singular rational curve with a double point. In both cases the irreducible components of $\alpha^{-1}(D_e)$ are $f^n(C_e)$ ($n\in\Z$), where $f$ is a generator of the automorphism group $\mathrm{Aut}_{X}\wt X\simeq\Z$.

For $e\ne e'\in \mathfrak{B}$ we see that the homology  classes $[C_e]$,  $[C_{e'}]$ are not congruent modulo $\mathrm{Aut}_{X}\wt X$, because the first terms in the decompositions of these classes are the lifts $\wt{e}$, $\wt{e'}$. Therefore $D_e\ne D_{e'}$.  In this way we obtain $n=b_2(X)$ rational curves in $X$, proving that $X$ is a GSS surface, by the main result in \cite{DOT}.
\end{proof}

\end{document}